\documentclass[a4paper,12pt, twoside]{amsart}
\usepackage{amsmath,amssymb,amsthm}
\usepackage{amscd}
\usepackage{color}
\usepackage{comment}
\usepackage{graphicx}	
\usepackage[all]{xy}
\usepackage{extarrows}
\usepackage{blkarray}
\usepackage{ulem}
\usepackage{tikz}

\theoremstyle{definition}
\newtheorem{dfn}{Definition}[section]
\newtheorem{thm}[dfn]{Theorem}
\newtheorem{prop}[dfn]{Proposition}
\newtheorem{lem}[dfn]{Lemma}
\newtheorem{cor}[dfn]{Corollary}

\theoremstyle{remark}

\numberwithin{equation}{section}

\newcommand\bK{{\mathbb{K}}}
\newcommand\bQ{{\mathbb{Q}}}
\newcommand\bR{{\mathbb{R}}}
\newcommand\bZ{{\mathbb{Z}}}

\newcommand\ch{\mathrm{ch}\,}

\newcommand\Gwt{\Gamma^{\mathbf{wt}}}
\newcommand\HK{H_{\mathbb{K}}}
\newcommand\Hom{\mathrm{Hom}}
\newcommand\inv{^{-1}}
\newcommand\Ker{\mathrm{Ker}}

\newcommand\La{\Lambda}
\newcommand\Mor{\mathrm{Mor}}

\newcommand\pa{\partial}
\newcommand\rot{\operatorname{rot}}
\newcommand\Sgone{\Sigma_{g,1}}

\title{$\mathbb{K}$-framings and $\mathbb{K}$-quadratic forms on surfaces}
\author{Nariya Kawazumi}
\address{Department of Mathematical Sciences, University of Tokyo, 3-8-1 Komaba, Meguro-ku, Tokyo 153-8914, Japan \texttt{e-mail:kawazumi@ms.u-tokyo.ac.jp}}
\date{} 



\dedicatory{in honor of Mikio Furuta on the occasion of his 65th birthday}

\begin{document}

\maketitle

\begin{abstract} We introduce the notions of $\mathbb{K}$-framings, based $\mathbb{K}$-framings and relative $\mathbb{K}$-framings of a compact connected oriented surface $\Sigma$ for any commutative ring $\mathbb{K}$ with unit, and a map which maps a based loop on $\Sigma$ to a homology class of its unit tangent bundle $U\Sigma$, which recovers Johnson's lifting in the case $\mathbb{K} = \mathbb{Z}/2$. This generalizes the correspondence between a quadratic form and a spin structure established by Johnson to any commutative ring $\mathbb{K}$ with unit. 
If the genus of $\Sigma$ is positive, we have a bijection between the set of $\mathbb{K}$-framings and the set of some twisted cocycles of the mapping class group of the surface $\Sigma$. 
Through this bijection, in the case where the boundary $\partial\Sigma$ is non-empty and connected, we discuss some relation between $\mathbb{K}$-framings and the extended first Johnson homomorphism. 
\end{abstract}

\bigskip
\begin{center}
Introduction
\end{center}

A spin structure on an oriented surface $\Sigma$ has played essential roles in various  classical contexts, and is important even now, which is illustrated also by the fact that the underlying structure of a super Riemann surface is a spin Riemann surface. In some classical context, it is given as a theta characteristic, which Johnson \cite{J80a} called a {\it quadratic form} on the surface $\Sigma$, while it is also a fiberwise double covering of the unit tangent bundle $U\Sigma$. It was Johnson \cite{J80a} to establish a clear correspondence between these two notions by introducing a lifting $H_1(\Sigma; \bZ/2) \to H_1(U\Sigma; \bZ/2)$. But the lifting is defined by taking a simple closed curve as a representative of each homology class. Hence it is a little bit complicated. 
In this paper, we construct a map $\pi_1(\Sigma) \to H_1(U\Sigma; \bZ)$, which recovers Johnson's lifting, by taking a generic immersed loop as a representative of each element of $\pi_1(\Sigma)$. This should clarify a geometric meaning of Johnson's lifting. \par
In this paper we consider only a compact connected oriented surface $\Sigma$. 
By the classification theorem, it is classified by its genus and the number of its boundary components, so that we denote $\Sigma = \Sigma_{g, n+1}$ if the genus equals $g \geq 0$ and the number of its boundary component is $n+1 \geq 0$. 
If $n \geq 0$, then the fundamental group $\pi_1(\Sigma)$ is free of rank $2g+n$. 
But, for a while, we confine ourselves to the surface $\Sigma_{g,1}$ of genus $g$ with one boundary component. Let $\Gamma_{g,1}$ be the mapping class group of $\Sigma_{g,1}$ fixing the boundary {\it pointwise}. 
For the surface $\Sigma_{g,1}$, the Poincar\'e duality gives a $\Gamma_{g,1}$-equivariant isomorphism $H_1(\Sigma_{g,1}; \bZ) = H^1(\Sigma_{g,1}; \bZ)$.  Morita \cite{Mo89} computed 
$H^1(\Gamma_{g,1}; H^1(\Sigma_{g,1}; \bZ)) = \bZ$ and named its generator $k \in H^1(\Gamma_{g,1}; H^1(\Sigma_{g,1}; \bZ))$, where he did not decide the sign of $k$. Prior to this computation, Earle \cite{E78} introduced a cocycle representing $k$ in the context of the Abel-Jacobi map, so that 
we call the class $k$ {\it the Earle class}. 
Chillingworth \cite{C72a, C72b} already constructed the class $k$ on the Torelli subgroup of $\Gamma_{g,1}$ using nowhere vanishing vector fields on $\Sigma_{g,1}$. 
On the whole mapping class group $\Gamma_{g,1}$, 
Furuta \cite[p.569]{Mo97} remarked that the class $k$ is the extension class of the extension of $\Gamma_{g,1}$-modules
$$
0 \to H^1(\Sigma_{g,1}; \bZ)  \to H^1(U\Sigma_{g,1}; \bZ) \to \bZ \to  0
\quad (\text{exact}).
$$
In other words, the class $k$ is the obstruction to take a $\Gamma_{g,1}$-invariant splitting of the extension. Trapp \cite{Trapp} also gave a similar observation. 
As was observed by Morita \cite{Mo93}, 
the first Johnson homomorphism \cite{J80b} on the Torelli subgroup is extended to a cocycle on the whole mapping class group with coefficients in $\La^3H_1(\Sigma_{g,1}; \bQ)$, and its reduction by the map $\La^3H_1(\Sigma_{g,1}; \bQ) \to H_1(\Sigma_{g,1}; \bQ)$ using the algebraic intersection number equals the class $k$ up to sign \cite[Remark 4.9]{Mo93}. 
One can construct explicit cocycles representing the extended first Johnson homomorphism by using symplectic expansions \cite{Mas12}
of the fundamental group $\pi_1(\Sigma_{g,1})$. 
Thus we like to find a framing corresponding to a given symplectic expansion. 
But the framing is not a usual framing of the  bundle $U\Sigma_{g,1}$, but a $\bQ$-framing in our sense. If $\bK= \bZ/2$, a $\bK$-framing is exactly a spin structure. Thus we consider $\bK$-framings for any commutative ring $\bK$ with unit. 
In this paper we compute explicitly the $\bK$-framing corresponding to any $\bK$-expansion of the fundamental group $\pi_1(\Sgone)$ satisfying a weaker condition than symplectic expansions (Lemma \ref{lem:C}, Corollary \ref{cor:ktauc}). 
The author \cite{Kaw06} constructed a $\bR$-symplectic expansion $\theta^{(C, P_0, v)}$ of $\pi_1(\Sgone)$ canonically associated with any triple $(C, P_0, v)$, where $C$ is a closed compact Riemann surface of genus $g$, $P_0 \in C$ is a point, and $v \in T_{P_0}C$ is a nonzero tangent vector. 
Our computation implies the $\bR$-framing corresponding to $\theta^{(C, P_0, v)}$ depends $(C, P_0, v)$ real analytically. 
It would be very interesting if one could find another geometric meaning of the $\bR$-framing. 
\par
\medskip
Any symplectic expansion provides the formality of the Goldman bracket \cite{G2} on the surface $\Sigma_{g,1}$ \cite{KK14}. 
In order to study the Turaev cobracket \cite{Tu1991} more precisely, 
our previous works \cite{AKKNa, AKKN18, 
AKKNb} introduced its framed version, and proved the formality of the framed Turaev cobracket for any compact connected oriented surface $\Sigma$ with non-empty boundary. 
Furuta's remark stated above made us aware of an essential role of framings in the formality. 
Later Hain \cite{Ha21} gave an alternative proof of the formality, and clarified the algebro-geometric meaning of the relation between the Turaev cobracket and framings. The target of the reduced extended first Johnson homomorphism is not $H^1(\Sigma; \bQ)$ but $H_1(\Sigma; \bQ) = H^1(\Sigma, \pa\Sigma; \bQ)$. They are different if the boundary $\pa\Sigma$ is not connected. 
Hence we introduce the notion of {\it relative} $\bK$-framings in \S\ref{sec:relf}, by modifying relative framings introduced by Randal-Williams \cite{RW14}. 
But the results in \S\ref{sec:exp} and \S\ref{sec:relf} require us to find an appropriate formulation of groupoid version of expansions for the surface $\Sigma$, which is our next problem. 
Here it should be remarked that Taniguchi \cite{Tan25b} studies a groupoid version of the framed based Turaev operation $\mu^\xi_r$ using non-commutative modular vector fields. 
Moreover, for any field $\bK$ of characteristic $0$, Taniguchi \cite{Tan25c} established a natural relation between genus $g$ Gozalez-Drinfel'd associators and $\bK$-framings for any compact connected oriented surface $\Sigma$ with non-empty boundary and any field $\bK$ of characteristic $0$, which covers all possible $\bK$-framings. 
Our desired groupoid expansions should provide the formality of the groupoid version of $\mu^\xi_r$. 
\par

Let $U\Sigma$ be the unit tangent bundle of the surface $\Sigma$, and 
$\Gamma$ its mapping class group fixing the boundary {\it pointwise}. 
A usual framing of $U\Sigma$ is an orientation-preserving trivialization, and can be regarded as a nowhere vanishing vector field on the surface $\Sigma$. 
Let $\bK$ be a commutative ring with unit. We call a splitting as $\bK$-modules of the extension 
$$
0 \to H^1(\Sigma; \bK) \to H^1(U\Sigma; \bK) \to \bK \to 0
\quad (\text{exact})
$$
of $\Gamma$-modules {\it a $\bK$-framing} on the surface $\Sigma$ (\S \ref{sec:frame}), 
and denote by $F(\Sigma; \bK)$ the set of $\bK$-framings on $\Sigma$.
We call the extension class of this extension {\it the $\bK$-Earle class} 
$k_\bK \in H^1(\Gamma; H^1(\Sigma; \bK))$ in \S\ref{sec:k}. 
Any $\bK$-framing $\xi \in F(\Sigma; \bK)$ defines a cocycle $k_\xi \in Z^1(\Gamma; H^1(\Sigma; \bK))$ representing $k_\bK$ in a natural way.
If we regard $k_\bK$ as a subset of $Z^1(\Gamma; H^1(\Sigma; \bK))$ in an obvious way, we have an isomorphism $F(\Sigma; \bK) \cong k_\bK$, $\xi \mapsto k_\xi$, if $g \geq 1$ (Corollary \ref{cor:kK}), while $k_\xi = 0$ for any $\xi \in F(\Sigma; \bK)$ if $g = 0$. 
In \S\ref{sec:pilift} we introduce a refinement 
$\La: \pi_1(\Sigma) \to  H_1(U\Sigma; \bZ)$ of Johnson's lifting.
The map $\La$ descends to a map on the quotient $\pi_1(\Sigma)/\Gwt_3\pi_1(\Sigma)$ (Theorem \ref{thm:wtG}). Here $\Gwt_3\pi_1(\Sigma)$ is the third term of the weight filtration of the fundamental group $\pi_1(\Sigma)$ of the surface $\Sigma$. 
Its composite $q_\xi := \xi\circ\La: \pi_1(\Sigma) \to \bK$ 
with a $\bK$-framing $\xi$ satisfies 
$$
q_\xi(\gamma_1\gamma_2)  = q_\xi(\gamma_1) + q_\xi(\gamma_2) + \gamma_1\cdot\gamma_2
$$
for any $\gamma_1, \gamma_2 \in \pi_1(\Sigma)$, where $\gamma_1\cdot\gamma_2$ is the algebraic intersection number. 
We call $q_\xi$ a {\it $\bK$-quadratic form} on the surface $\Sigma$, 
which is exactly the constant term of the framed based Turaev operation $\mu^\xi_r$ \cite{Tu1979, AKKNa, AKKN18, 
AKKNb, Tan25b}, and can be interpreted as a writhe of an immersed loop normalized by the $\bK$-framing $\xi$. 
We have a natural isomorphism of $H^1(\Sigma; \bK)$-torsors between the set of $\bK$-framings and that of $\bK$-quadratic forms (Proposition \ref{prop:FQ}). 
If $\bK = \bZ/2$, $q_\xi$ descends to the quadratic form on $H_1(\Sigma; \bZ/2)$ associated with the spin structure $\xi$ introduced by Johnson \cite{J80a}. 
For a closed connected oriented surface $\Sigma_{g,0}$ of genus $g \geq 0$, 
by the inclusion isomorphism $H_1(\Sgone; \bK) = H_1(\Sigma_{g,0}; \bK)$, this isomorphism recovers Johnson's result \cite[Theorem 3.A, p.371]{J80a}.
\par
In order to modify the notion of a relative framing given by Randal-Williams \cite{RW14}, we choose a finite subset $P \subset \pa\Sigma$ whose inclusion homomorphism $\pi_0(P) \to \pi_0(\pa\Sigma)$ is an isomorphism, and consider an intermediate pair $(\Sigma, P)$ between $\Sigma$ and $(\Sigma, \pa\Sigma)$. 
We call a $\bK$-framing of $(\Sigma, P)$ a {\it based} $\bK$-framing in \S\ref{sec:frame}. 
We construct a refinement of Johnson's lifting for $(\Sigma, P)$ in \S\ref{sec:lift}, and establish an isomorphism between the set $F(\Sigma, P; \bK)$ of based $\bK$-framings and that of $\bK$-quadratic forms on $(\Sigma, P)$ (Proposition \ref{prop:FQ}).
Similar to the $\bK$-Earle class $k_\bK$, we have an extension class $k^\natural_\bK \in H^1(\Gamma; H^1(\Sigma, P; \bK))$ defined by any element of $F(\Sigma, P; \bK)$. For their own study on the moduli space of twisted holomorphic $1$-forms on Riemann surfaces, Apisa and Salter \cite[Theorem A.4]{ApisaSalter} computed $H^1(\Gamma; H^1(\Sigma, \pa\Sigma; \bK))$ and proved a natural isomorphism $F(\Sigma, P; \bK) \cong k^\natural_\bK$ for $\Sigma = \Sigma_{g,n+1}$ with $g \geq 3$ and $n+1 \geq 3$. \par
In \S\ref{sec:relf}, we define a {\it relative} $\bK$-framing of $\Sigma$ by a based $\bK$-framing together with data $\rho \in \bK^{\pi_0(\pa\Sigma)}$ of the rotation numbers along all the boundary components. 
It is justified by the fact that the space of trivializations of $U\Sigma$ restricted to each boundary component is a fibration over $S^1$ with fiber $\Omega S^1\simeq \bZ$. The set $F(\Sigma, P, \rho; \bK)$ of relative $\bK$-framings of $\Sigma$ with respect to $\rho$ is non-empty if and only if $\rho$ respects the Poincar\'e-Hopf formula (Lemma \ref{lem:nonempty}). If $F(\Sigma, P, \rho; \bK)$ is non-empty, then any of its elements defines a cohomology class $k_{\rho, \bK} \in H^1(\Gamma; H^1(\Sigma, \pa\Sigma; \bK))$ which is a lift of the $\bK$-Earle class $k_\bK \in H^1(\Gamma; H^1(\Sigma; \bK))$. 
It is our next problem to compute explicitly a relative $\bK$-framing
associated with a given $\bK$-expansion of the fundamental groupoid $\Pi\Sigma\vert_P$ restricted to $P$ similarly to Corollary \ref{cor:ktauc}. \par
The quotient of the set $F(\Sigma; \bK)$ by the action of the mapping class group $\Gamma$ was partially discussed in \cite{RW14, Kaw18}. 
Moreover the action of the hyperelliptic mapping class group on the set of spin structures on a closed surface was discussed also in \cite{Wang1}. 
But it is our future problem how to describe the quotient $F(\Sigma; \bK)/\Gamma$ for any $\Sigma$ and $\bK$. \par

In contrast to contravariant coefficients \cite{S}, it is still unsolved to compute the stable twisted cohomology groups of the mapping class groups with {\it covariant} coefficients. In \cite{KS1, KS2}, Souli\'e and the author computed some of them, where the class $k_\bQ$ plays an essential role. \par

\bigskip
\noindent
{\bf Convention}: Let $\gamma_1$ and $\gamma_2$ be paths on $\Sigma$, and suppose the endpoint of $\gamma_1$ equals the start point of $\gamma_2$. Then we denote by the concatenation $\gamma_1\gamma_2$ the path traversing $\gamma_1$ first, then $\gamma_2$. \par


\bigskip
\noindent
{\bf Acknowledgement}: 
The author would like to thank 
Anton Alekseev,
Yusuke Kuno, 
Sergei Merkulov, 
Florian Naef,
Paul Norbury, 
Kento Sakai, 
Arthur Souli\'e
and Toyo Taniguchi
for helpful discussions, some of which go back to about 10 years ago. 
Moreover he would like to express his gratitude to Nick Salter for letting him know Apisa-Salter \cite{ApisaSalter} and other interesting works.  
%
%
The present research is supported in part by the grants JSPS KAKENHI 
18K03283 
, 19H01784
, 20H00115
, 22H01125
, 22H01120, 23K22391 
and the joint project between OIST (Hikami Unit) and University of Tokyo.

\tableofcontents

%
%

\section{A refinement of Johnson's lifting}\label{sec:lift}
Let $\Sigma$ be a compact connected oriented $C^\infty$ surface. 
By the classification theorem of surfaces, it is classified by the genus $g \geq 0$ and the number $n+1 \geq 0$ of the boundary components, where $n \geq -1$. We also denote $\Sigma = \Sigma_{g,n+1}$. 
If $\pa\Sigma \neq \emptyset$, then we label each of the boundary components as $\pa_j\Sigma$ with $0 \leq j \leq n$. Moreover we choose a basepoint $\ast$ on the $0$-th boundary $\pa_0\Sigma$ and denote $\pi := \pi_1(\Sigma, *)$, which is a free group of rank $2g+n$. 
\par
The unit tangent bundle $\varpi: U\Sigma \to \Sigma$ of the surface $\Sigma$ is defined by using a Riemannian metric on $\Sigma$, but also by the quotient $T^\times \Sigma/\bR_{>0}$ of the complement  $T^\times \Sigma$ of the zero section in the tangent bundle $T\Sigma$ by the scalar multiplication of the positive real numbers $\bR_{>0}$. 
Any diffeomorphism $\varphi$ of $\Sigma$ acts on the space $U\Sigma$  by its differential $d\varphi$ through the latter definition.
We denote the unit interval in $\bR$ by $I = [0,1]$, its interior by $I^\circ := \mathopen]0,1\mathclose[$, and its boundary by $\pa I := \{0,1\}$. 
Any $C^\infty$ immersion $\gamma: I \to \Sigma$ 
induces the velocity vector $\overset\cdot\gamma: t\in I \mapsto (d\gamma)\left(\frac{\pa}{\pa t}\right)_t \bmod \bR_{>0} \in U\Sigma$. 
\par
Johnson \cite{J80a} introduced a canonical lifting 
$H_1(\Sigma; \bZ/2) \to H_1(U\Sigma; \bZ/2)$ of $\varpi_*$
 to identify a quadratic form on $H_1(\Sigma; \bZ/2)$ with a spin structure on $\Sigma$. 
In this section we introduce a refinement of Johnson's lifting which is applicable for $\bK$-framings on $\Sigma$ with $\pa\Sigma \neq \emptyset$ for any commutative ring $\bK$ with unit. See \S3 for the definition of a $\bK$-framing.  If $\bK = \bZ/2$, our refinement recovers Johnson's lifting also for any closed surface $\Sigma$. 
\par
Suppose $\pa\Sigma \neq \emptyset$, and choose a finite subset $P \subset \pa\Sigma$ with $\ast \in P$. Here we do {\it not} assume that the inclusion homomorphism $\pi_0(P) \to \pi_0(\pa\Sigma)$ is an isomorphism. 
Let $\tilde P \subset U\Sigma$ be the set of the negative tangent vectors along $\pa \Sigma$ at all points in $P$. 
The map $\varpi$ restricted to $\tilde P$ is a bijection onto $P$. 
We denote the negative tangent vector at $\bullet \in P$ by $v_{\bullet}\in \tilde P$. 
Our refinement will be defined on the fundamental groupoid $\Pi\Sigma\vert_P$ restricted to the subset $P$,
whose object set is $P$, and the set of morphism $\Pi\Sigma(\bullet_0, \bullet_1)$, $\bullet_0, \bullet_1 \in P$, is the based homotopy set of paths from $\bullet_0$ to $\bullet_1$ on $\Sigma$. We may take a generic immersion $\gamma: (I, 0, 1) \to (\Sigma, \bullet_0, \bullet_1)$ satisfying $\overset\cdot\gamma(0) = v_{\bullet_0}$, $\overset\cdot\gamma(1) = v_{\bullet_1}$ and $\gamma(I^\circ) \subset \Sigma\setminus \pa \Sigma$ as a representative for each element in $\Pi\Sigma(\bullet_0, \bullet_1)$. See Figure 1. Here `generic' means $\gamma\vert_{I^\circ}$ allows only transverse double points as its singularities.  
We call such a generic immersion as a {\it generic representative} of a morphism in $\Pi\Sigma\vert_P$. Here and throughout this paper, by abuse of notation, we denote both of a generic immersion and its homotopy class by the same symbol $\gamma$. 
\par
\begin{figure}[h]
\begin{center}
\begin{tikzpicture}
\draw[thick] (-1.5,0) -- (1.5,0); 
\node at (1,0) {$>$}; 
\node at (-1.0,0.5) {$\circlearrowleft$}; 
\node at (1.0,0.5) {$\Sigma$}; 
\node at (2.0,0.0) {$\partial_{j_1}\Sigma$}; 
\node at (-0.5,0) {$<$};
\node at (-0.3,-0.4) {\small $v_{\bullet_1}$}; 
\coordinate (a_1) at (0,0) node at (a_1) [below] {$\bullet_1$}; 
\fill (a_1) circle (2pt) ;
\draw[thick] (a_1) to [out= 0 , in=270](0.4, 1.0); 
\draw[thick] (0.4, 1.0) -- (0.4, 1.6) node [midway] {$\vee$};
\draw[thick, dashed] (0.4, 1.6) to [out= 90, in=0](-1.0, 3.2); 
\node at (0.8, 1.3) {$\gamma$}; 
\coordinate (b) at (-2.0, 1.7);
\draw[thick] (b) to [out= 180 , in=270](-2.4,2.3); 
\draw[thick] (b) to [out= 0 , in=270](-1.6, 2.3); 
\draw[thick] (-2.4,2.3) to [out= 90 , in=180] (-1.0, 3.2);
\draw[thick] (-1.6, 2.3) to [out= 90 , in=0] (-3.0, 3.2);
\node at (-2.0,1.7) {$<$};
\fill (-2.0,2.95) circle (2pt) ;
\node at (-1.6,2.95) {$+$};
\draw[thick] (-5.5,0) -- (-3,0); 
\node at (-3.3,0) {$>$}; 
\node at (-5.0,0.5) {$\circlearrowleft$}; 
\node at (-2.5,0.0) {$\partial_{j_0}\Sigma$}; 
\node at (-4.5,0) {$<$};
\node at (-4.3,-0.4) {\small $v_{\bullet_0}$}; 
\coordinate (a_0) at (-4.0,0) node at (a_0) [below] {$\bullet_0$}; 
\fill (a_0) circle (2pt) ;
\draw[thick] (a_0) to [out= 180 , in=270](-4.4, 1.0); 
\draw[thick] (-4.4, 1.0) -- (-4.4, 1.6) node [midway] {$\wedge$};
\draw[thick, dashed] (-4.4, 1.6) to [out= 90, in=180](-3.0, 3.2); 
\node at (-4.8, 1.3) {$\gamma$}; 
\end{tikzpicture}
\end{center}
\caption{a generic representative $\gamma$ of an element of $\Pi\Sigma(\bullet_0, \bullet_1)$}
\end{figure}
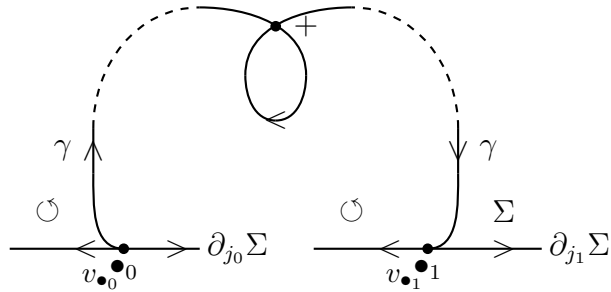
We denote the set of all self-intersection points of $\gamma\vert_{I^\circ}$ by $\Gamma_\gamma$. For $p \in \Gamma_\gamma$, there exist unique $0 < t^p_1 < t^p_2 < 1$ such that $\gamma(t^p_1) = \gamma(t^p_2) = p$ from the genericity of the immersion $\gamma$. 
The local intersection number $\varepsilon_p \in \{\pm1\}$ is defined by comparing the orientation of $\Sigma$ and the oriented basis $(\overset\cdot\gamma(t^p_1),\overset\cdot\gamma(t^p_2))$ of $T_p\Sigma$. 
We denote $\Gamma^+_\gamma := \{p \in \Gamma_\gamma; \,\,\varepsilon_p = +1\}$ and $\Gamma^-_\gamma := \{p \in \Gamma_\gamma; \,\,\varepsilon_p = -1\}$. 
Moreover the fundamental class $[I, \pa I] \in H_1(I, \pa I; \bZ)$ is mapped to the homology class  
$$
\overrightarrow{\gamma} := \overset\cdot\gamma_*[I, \pa I] \in H_1(U\Sigma, \tilde P; \bZ)
$$
by the velocity vector $\overset\cdot\gamma: (I, \pa I) \to (U\Sigma, \tilde P)$.
\par
Let $\iota: S^1 \hookrightarrow U\Sigma$ be the inclusion map of the fiber of $\varpi: U\Sigma \to \Sigma$. Following Johnson \cite{J80a}, we denote $z := \iota_*[S^1] \in H_1(U\Sigma; \bZ)$. 
Then we define a map 
$\La: \Pi\Sigma(\bullet_0, \bullet_1) \to H_1(U\Sigma, \tilde P; \bZ)$
for any $\bullet_0, \bullet_1\in P$ 
by 
\begin{equation}\label{eq:tildeL}
\La(\gamma) := \overrightarrow{\gamma} + \left(\sharp\Gamma^+_{\gamma} - 
\sharp\Gamma^-_{\gamma} + 1\right)z \in H_1(U\Sigma, \tilde P; \bZ)
\end{equation}
for a generic representative $\gamma: (I, 0, 1) \to (\Sigma, \bullet_0, \bullet_1)$ of any element in $\Pi\Sigma\vert_P(\bullet_0, \bullet_1)$. In order to prove the map $\La$ is well-defined, it suffices to check that $\La$ is invariant under the three local moves $(\omega1)$, $(\omega2)$ and $(\omega3)$ in \cite[5.6 Lemma]{G2}. 
The invariance under $(\omega2)$(birth-death of bigons) and $(\omega3)$ (jumping over a double point) is clear, since $\La$ is a regular homotopy invariant. If $p \in \Gamma_\gamma$ is the vertex of a monogon (or a fishtail), the monogon contributes $-\varepsilon_pz$ to the homology class $\overrightarrow{\gamma}$. Hence $\La$ is also invariant under $(\omega1)$ (birth-death of monogons). 
This proves the map $\La$ is well-defined.\par
We denote by $\Mor(\Pi\Sigma\vert_P)$ the set of all morphisms in the groupoid $\Pi\Sigma\vert_P$
$$
\Mor(\Pi\Sigma\vert_P) = {\coprod}_{(\bullet_0, \bullet_1) \in P\times P}\Pi\Sigma(\bullet_0, \bullet_1).
$$
But the map $\La: \Mor(\Pi\Sigma\vert_P) \to H_1(U\Sigma, \tilde P; \bZ)$ is not a functor if we regard the homology group $H_1(U\Sigma, \tilde P; \bZ)$ as a category with a single object. 
But it satisfies some product formula (Lemma \ref{lem:product}), where the algebraic intersection number between paths connecting points in $P$ appears. 
To fix our convention for the algebraic intersection number, 
we take a disjoint family $\{\nu_\bullet\}_{\bullet \in P}$ of orientation-reversing embeddings $\nu_\bullet: [-1, 1] \to \pa\Sigma$ with $\nu_\bullet(0) = \bullet$, 
denote $\bullet^\pm := \nu_\bullet(\pm1) \in \pa\Sigma$. 
Then we define the algebraic intersection number $\gamma_1\cdot\gamma_2$ for any $\gamma_1 \in \Pi\Sigma(\bullet_0, \bullet_1)$ and $\gamma_2 \in \Pi\Sigma(\bullet_1, \bullet_2)$ by the intersection product 
$$
\gamma_1\cdot\gamma_2 := \left[\left(\nu_{\bullet_0}\vert_{[0,1]}\right)\inv\gamma_1\left(\nu_{\bullet_1}\vert_{[0,1]}\right)\right]\cdot 
 \left[\left(\nu_{\bullet_1}\vert_{[-1,0]}\right)\gamma_2\left(\nu_{\bullet_2}\vert_{[-1,0]}\right)\inv\right] \in \bZ
$$
between the homology groups $H_1(\Sigma, A; \bZ)$ and $H_1(\Sigma, B; \bZ)$, 
where $A$ and $B \subset \pa \Sigma$ satisfy $\bullet^+ \in A$ and $\bullet^- \in B$ for any $\bullet \in P$, $A\cup B = \pa\Sigma$ and $\pa A = \pa B = A\cap B$. 
This is justified by the Poincar\'e-Lefschetz duality $H^1(\Sigma, A; \bZ) \cong H_{1}(\Sigma, B; \bZ)$. See, for example, \cite[Theorem 3.43, p.254]{Hat02}. 
If $\bullet_0 = \bullet_1$ or $\bullet_1 = \bullet_2$, then the number $\gamma_1\cdot\gamma_2$ does not depend on our convention. \par
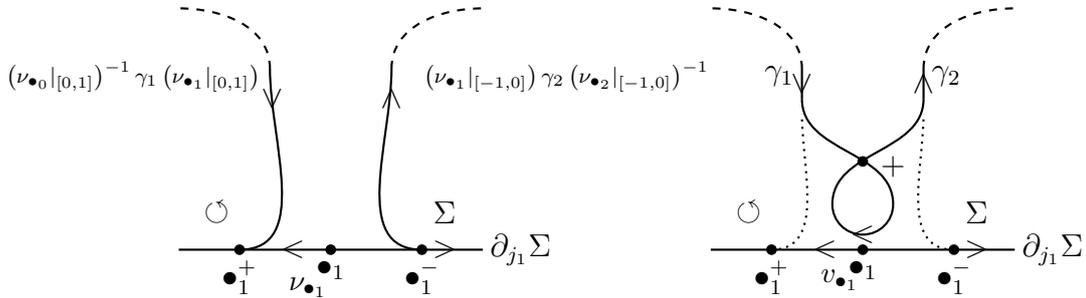
\begin{figure}[h]
\begin{center}
\begin{tikzpicture}
\begin{scope}
\draw[thick] (-2,0) -- (2,0); 
\node at (1.5,0) {$>$}; 
\node at (-1.5,0.5) {$\circlearrowleft$}; 
\node at (1.5,0.5) {$\Sigma$}; 
\node at (2.5,0.0) {$\partial_{j_1}\Sigma$}; 
\node at (-0.5,0) {$<$};
\node at (-0.3,-0.5) {\small $\nu_{\bullet_1}$}; 
\coordinate (a_1) at (0,0) node at (a_1) [below] {$\bullet_1$}; 
\fill (a_1) circle (2pt) ;
\draw[thick, dashed] (0.8, 2.4) to [out= 90, in=180](2.0, 3.2); 
\draw[thick, dashed] (-0.8, 2.4) to [out= 90, in=0](-2.0, 3.2); 
\node at (3.1, 2.3) {\tiny $\left(\nu_{\bullet_1}\vert_{[-1,0]}\right)\gamma_2\left(\nu_{\bullet_2}\vert_{[-1,0]}\right)\inv$}; 
\node at (-2.6, 2.3) {\tiny $\left(\nu_{\bullet_0}\vert_{[0,1]}\right)\inv\gamma_1\left(\nu_{\bullet_1}\vert_{[0,1]}\right)$}; 
\coordinate (a-) at (1.2,0) node at (a-) [below] {$\bullet^-_1$}; 
\fill (a-) circle (2pt) ;
\draw[thick] (a-) to [out= 180 , in=270](0.8,2.4); 
\coordinate (a+) at (-1.2,0) node at (a+) [below] {$\bullet^+_1$}; 
\fill (a+) circle (2pt) ;
\draw[thick] (a+) to [out= 0 , in=270](-0.8,2.4); 
\node at (0.8, 2.0) {$\wedge$};
\node at (-0.77, 2.0) {$\vee$};
\end{scope}
\begin{scope}[xshift=7cm]
\draw[thick] (-2,0) -- (2,0); 
\node at (1.5,0) {$>$}; 
\node at (-1.5,0.5) {$\circlearrowleft$}; 
\node at (1.5,0.5) {$\Sigma$}; 
\node at (2.5,0.0) {$\partial_{j_1}\Sigma$}; 
\node at (-0.5,0) {$<$};
\node at (-0.3,-0.4) {\small $v_{\bullet_1}$}; 
\coordinate (a_1) at (0,0) node at (a_1) [below] {$\bullet_1$}; 
\fill (a_1) circle (2pt) ;
\coordinate (au) at (0,0.2); 
\draw[thick] (au) to [out= 180 , in=270](-0.4,0.6); 
\draw[thick] (-0.4, 0.6) to [out= 90 , in=270] (0.8, 2.0);
\draw[thick] (0.4, 0.6) to [out= 90 , in=270] (-0.8, 2.0);
\node at (0,0.2) {$<$};
\draw[thick] (0.8,2.0) -- (0.8,2.4) node [midway] {$\wedge$};
\draw[thick] (-0.8,2.0) -- (-0.8,2.4) node [midway] {$\vee$};
\draw[thick, dashed] (0.8, 2.4) to [out= 90, in=180](2.0, 3.2); 
\draw[thick] (au) to [out= 0 , in=270](0.4, 0.6); 
\draw[thick, dashed] (-0.8, 2.4) to [out= 90, in=0](-2.0, 3.2); 
\node at (1.1, 2.3) {$\gamma_2$}; 
\node at (-1.1, 2.3) {$\gamma_1$}; 
\fill (0,1.17) circle (2pt) ;
\node at (0.4,1.1) {$+$}; 
\coordinate (a-) at (1.2,0) node at (a-) [below] {$\bullet^-_1$}; 
\fill (a-) circle (2pt) ;
\draw[thick, dotted] (a-) to [out= 180 , in=270](0.8,1.8); 
\coordinate (a+) at (-1.2,0) node at (a+) [below] {$\bullet^+_1$}; 
\fill (a+) circle (2pt) ;
\draw[thick, dotted] (a+) to [out= 0 , in=270](-0.8,1.8); 
\end{scope}
\end{tikzpicture}
\end{center}
\caption{Computation of the algebraic intersection number and the concatenation of $\gamma_1$ and $\gamma_2$}
\end{figure}
Applying a regular homotopy, we have a generic representative $I\to \Sigma$ for the concatenation $\gamma_1\gamma_2$ as in the right hand side in Figure 2. 
If we denote it also by $\gamma_1\gamma_2$, we have 
\begin{equation}\label{eq:shsum}
\sharp\Gamma^+_{\gamma_1\gamma_2} - 
\sharp\Gamma^-_{\gamma_1\gamma_2} 
= \sharp\Gamma^+_{\gamma_1} - 
\sharp\Gamma^-_{\gamma_1} 
+\sharp\Gamma^+_{\gamma_2} - 
\sharp\Gamma^-_{\gamma_2} +1
+\gamma_1\cdot \gamma_2\in \bZ,
\end{equation}
and 
\begin{equation}\label{eq:arsum}
\overrightarrow{\gamma_1\gamma_2}
= \overrightarrow{\gamma_1}
+ \overrightarrow{\gamma_2} \in H_1(U\Sigma, \tilde P; \bZ).
\end{equation}
Consequently we obtain the following product formula.
\begin{lem}\label{lem:product} For any $\gamma_1 \in \Pi\Sigma(\bullet_0, \bullet_1)$ and $\gamma_2 \in \Pi\Sigma(\bullet_1, \bullet_2)$, we have 
$$
\La(\gamma_1\gamma_2) = \La(\gamma_1)+\La(\gamma_2) + (\gamma_1\cdot\gamma_2)z \in H_1(U\Sigma, \tilde P; \bZ).
$$
\end{lem}
\par
As will be shown in the next section, our construction is a refinement of Johnson's lifting, 
where $P$ consists only of a single point and the coefficients $\bZ$ is replaced by $\bZ/2$.
\par

\section{The lifting restricted to the fundamental group}\label{sec:pilift}

In this section we confine ourselves to the case where $P$ consists of a single point $\ast \in \pa_0\Sigma$. In particular, any boundary component other than $\pa_0\Sigma$ contains no elements of the set $P$. 
Then our lifting $\La$ is a map from the fundamental group $\pi = \pi_1(\Sigma, \ast) = \Pi\Sigma\vert_P$ to the homology group $H_1(U\Sigma; \bZ) = H_1(U\Sigma, \tilde P; \bZ)$. From the definition, $\varpi_*\circ \La: \pi \to H_1(\Sigma; \bZ)$ equals the abelianization map $\pi \to \pi/\Gamma_2\pi = H_1(\Sigma; \bZ) = H_1(\Sigma, P; \bZ)$. Hence $\La$ and $L$ stated below are liftings of the abelianization map. \par
\begin{lem}\label{lem:commutator} For any $\gamma_1, \gamma_2 \in \pi$, we have
$$
\La(\gamma_1\gamma_2{\gamma_1}\inv{\gamma_2}\inv)
= 2(\gamma_1\cdot\gamma_2)z \in H_1(U\Sigma; \bZ).
$$
\end{lem}
\begin{proof}
From Lemma \ref{lem:product}, 
we obtain 
$$
\La(\gamma_1\gamma_2{\gamma_1}\inv{\gamma_2}\inv)
= \La(\gamma_1\gamma_2{\gamma_1}\inv{\gamma_2}\inv) +\La(\gamma_2\gamma_1)- \La(\gamma_2\gamma_1)
= \La(\gamma_1\gamma_2) - \La(\gamma_2\gamma_1) = 2(\gamma_1\cdot\gamma_2)z.
$$
The second equality comes from $(\gamma_1\gamma_2{\gamma_1}\inv{\gamma_2}\inv)\cdot(\gamma_2\gamma_1) = 0$. This proves the lemma.
\end{proof}

In order to show that $\La$ recovers Johnson's lifting, 
we denote the $\bmod\,2$ reduction of the map $\La$ by $\La_2: \pi \to H_1(U\Sigma; \bZ/2)$. Then Lemma \ref{lem:commutator} implies $\La_2\vert_{[\pi, \pi]} = 0$. 
Since $\gamma\cdot\delta =0$ for any $\gamma \in \pi$ and $\delta \in [\pi, \pi]$, $\La_2$ descends to a map $H_1(\Sigma; \bZ/2) = (\pi/[\pi, \pi])\otimes_{\bZ}\bZ/2 \to H_1(U\Sigma; \bZ/2)$, which we denote by the same symbol $\La_2$. 
From \eqref{eq:tildeL}, if $x \in H_1(\Sigma; \bZ/2)$ is represented by a simple closed curve, then we have 
$$
\La_2(x) = \overrightarrow{x} + z \in H_1(U\Sigma; \bZ/2),
$$
where $\overrightarrow{x}$ is the homology class of the velocity vector of the simple closed curve as before. For any $x_1$ and $x_2 \in H_1(\Sigma; \bZ/2)$, we have 
$$
\La_2(x_1+x_2) = \La_2(x_1) + \La_2(x_2) + (x_1\cdot x_2)z \in H_1(U\Sigma; \bZ/2)
$$
from Lemma \ref{lem:product}. These formulae mean that the map $\La_2$ is exactly the same as Johnson's lifting for any compact oriented surface with non-empty boundary.
If $\Sigma$ is a closed oriented surface of genus $g$, the inclusion $\Sigma_{g,1} \hookrightarrow \Sigma$ induces an isomorphism $H_1(\Sigma_{g,1}; \bZ/2) \cong H_1(\Sigma; \bZ/2)$, so that the composite of the map $\La_2$ and the inclusion homomorphism \\
$H_1(U\Sigma_{g,1}; \bZ/2) \to H_1(U\Sigma; \bZ/2)$ coincides with Johnson's lifting \cite{J80a}. \par
\bigskip
If $g = 0$, the map $\La$ descends to $\pi/[\pi, \pi] = H_1(\Sigma_{0,n+1}; \bZ)$ by Lemma \ref{lem:product}. In general, the map $\La$ itself fits to the weight filtration of the fundamental group $\pi = \pi(\Sigma, \ast)$ for $\Sigma = \Sigma_{g,n+1}$. For details of the weight filtration, see, for example, \cite[\S\S 2.6, 3.1 and 3.2]{AKKNb}. 
We define a surface $\overline{\Sigma}$ by gluing $n$ copies of $\Sigma_{1,1}$ along the boundary components $\pa_j\Sigma$,  $1 \leq j \leq n$, except $\pa_0\Sigma$
$$
\overline{\Sigma} := \Sigma \cup_{\pa\Sigma\setminus \pa_0\Sigma}\left({\bigcup}^n_{j=1}\Sigma_{1,1}\right).
$$
Then $\overline{\Sigma}$ is a compact surface of genus $g+n$ with one boundary component $\pa_0\Sigma$. The fundamental group $\overline{\pi} := \pi_1(\overline{\Sigma}, \ast)$ is a free group of rank $2(g+n)$. We denote the inclusion homomorphism by $\imath: \pi \to \overline{\pi}$. 
The lower central series $\{\Gamma_k\overline{\pi}\}_{k\geq 1}$ of the group $\overline{\pi}$ is defined by $\Gamma_1\overline{\pi} := \overline{\pi}$, and $\Gamma_{k+1}\overline{\pi} := [\overline{\pi}, \Gamma_{k}\overline{\pi}]$ for $k \geq 1$. 
The weight filtration $\{\Gwt_k\pi\}_{k\geq 1}$ is defined to be the pull-back of $\{\Gamma_k\overline{\pi}\}_{k\geq 1}$ by the inclusion homomorphism $\imath$
$$
\Gwt_k\pi := \imath\inv(\Gamma_k\overline{\pi}) \subset \pi, \quad k \geq 1,
$$
which is a central filtration. \par
\begin{thm}\label{thm:wtG} The lifting $\La: \pi \to H_1(U\Sigma; \bZ)$ descends to a map $\pi/\Gwt_3\pi \to H_1(U\Sigma; \bZ)$. 
\end{thm}

In order to prove Theorem \ref{thm:wtG}, we need some computation of homology groups.
For a while, we omit the coefficients $\bZ$ of homology groups.
Since $\overline{\Sigma}$ has non-empty boundary $\pa_0\Sigma$, we have a section $s: \overline{\Sigma} \to U\overline{\Sigma}$ of the bundle $U\overline{\Sigma}$. 
\begin{lem}\label{lem:kernel} The kernel of the inclusion homomorphism $i_*: H_1(U\Sigma) \to H_1(U\overline{\Sigma})$ is generated by $\{s_*[\pa_j\Sigma]; \,\, 1 \leq j \leq n\}$. 
\end{lem}
\begin{proof}
By excision, we have $H_*(U\overline{\Sigma}, U\Sigma) \cong H_*(U\Sigma_{1,1}, U\Sigma_{1,1}\vert_{\pa\Sigma_{1,1}})^{\oplus n}$ through the inclusion $(U\Sigma_{1,1}, U\Sigma_{1,1}\vert_{\pa\Sigma_{1,1}})^{\amalg n} \hookrightarrow (U\overline{\Sigma}, U\Sigma)$. 
The section $s$ restricts a section on each of the bundles $U\Sigma_{1,1}$'s. 
Hence it suffices to show that the kernel of the inclusion homomorphism $H_1(U\Sigma_{1,1}\vert_{\pa\Sigma_{1,1}}) \to H_1(U\Sigma_{1,1})$ is generated by $s_*[\pa\Sigma_{1,1}]$. The section $s$ induces an isomorphism of bundles $U\Sigma_{1,1} \cong \Sigma_{1,1}\times S^1$. Together with the K\"unneth formula it implies an isomorphism 
$H_2(U\Sigma_{1,1}, U\Sigma_{1,1}\vert_{\pa\Sigma_{1,1}}) \cong (H_1(\Sigma_{1,1}, \pa\Sigma_{1,1})\otimes H_1(S^1)) \oplus (H_2(\Sigma_{1,1}, \pa\Sigma_{1,1})\otimes H_0(S^1))$. 
The first direct summand is hit by the inclusion homomorphism $j_*: H_2(U\Sigma_{1,1}) \to H_2(U\Sigma_{1,1}, U\Sigma_{1,1}\vert_{\pa\Sigma_{1,1}})$, 
since the section $s$ induces an isomorphism $H_2(U\Sigma_{1,1}) \cong H_1(\Sigma_{1,1})\otimes H_1(S^1)$, 
while the connecting homomorphism $\partial_*: H_2(U\Sigma_{1,1}, U\Sigma_{1,1}\vert_{\pa\Sigma_{1,1}}) \to H_1(U\Sigma_{1,1}\vert_{\pa\Sigma_{1,1}})$ maps the second direct summand to the submodule generated by $s_*[\pa\Sigma_{1,1}]$. This proves Lemma \ref{lem:kernel}. 
\end{proof}

\begin{proof}[Proof of Theorem \ref{thm:wtG}]  Since the algebraic intersection number is preserved under the inclusion homomorphism $\imath: \pi \to \overline{\pi}$, we have $\gamma_1\cdot \gamma_2 = 0$ for any $\gamma_1\in\pi$ and $\gamma_2 \in \Gwt_2\pi$, 
so that $\La(\gamma_1\gamma_2) = \La(\gamma_1) + \La(\gamma_2)$. 
Hence its suffices to show that $\La$ vanishes on $\Gwt_3\pi$. \par
Let $\overline{\La}: \overline{\pi} \to H_1(U\overline{\Sigma}; \bZ)$ be the map $\La$ for the surface $\overline{\Sigma}$. 
We have $\overline{\La}\circ\imath = i_*\circ \La: \pi \to H_1(U\overline{\Sigma}; \bZ)$. 
For any $\overline{\gamma}_1 \in \overline{\pi}$ and $\overline{\gamma}_2 \in \Gamma_2\overline{\pi}$ we have $\overline{\La}(\overline{\gamma}_1\overline{\gamma}_2{\overline{\gamma}_1}\inv{\overline{\gamma}_2}\inv) = 2\overline{\gamma}_1\cdot \overline{\gamma}_2 = 0$ from Lemma \ref{lem:commutator}. Hence $\overline{\La}$ vanishes on $\Gamma_3\overline{\pi}$ and so on $\Gwt_3\pi$. \par
Now we take free generating systems $\{\overline{\alpha}_i, \overline{\beta}_i\}^{g+n}_{i=1}$ and $\{\alpha_i, \beta_i\}^g_{i=1}\cup \{\delta_j\}^n_{j=1}$ of the groups $\overline{\pi}$ and $\pi$, respectively, such that $\imath(\alpha_i) = \overline{\alpha}_i$, $\imath(\beta_i) = \overline{\beta}_i$ for any $1 \leq i \leq g$, and that $\imath(\delta_j) = \overline{\alpha}_{g+j}\overline{\beta}_{g+j}{\overline{\alpha}_{g+j}}\inv{\overline{\beta}_{g+j}}\inv$ represents the boundary loop $\pa_j\Sigma$ for any $1 \leq j \leq n$.  
Let $\gamma \in \Gwt_3\pi$. Then, since $\overline{\La}(\gamma) = 0$, 
we have $\La(\gamma) = \sum^n_{j=1}\nu_js_*[\pa_j\Sigma]$ for some $\nu_j \in \bZ$
from Lemma \ref{lem:kernel}. Applying $\varpi: U\Sigma \to \Sigma$, we have 
$[\gamma] = \sum^n_{j=1}\nu_j[\pa_j\Sigma] \in H_1(\Sigma)$, which means we have 
$\gamma = {\delta_1}^{\nu_1}{\delta_2}^{\nu_2}\cdots {\delta_n}^{\nu_n}\gamma'$ for some $\gamma' \in \Gamma_2\pi$. 
Denote by $\overline{H} = H_1(\overline{\Sigma}; \bQ)$ and take an expansion $\overline{\theta} = \sum^\infty_{m=0}\overline{\theta}_m: \overline{\pi} \to \widehat{T}(\overline{H}) = \prod^\infty_{m=0}\overline{H}^{\otimes m}$, 
$\overline{\theta}_m: \overline{\pi} \to \overline{H}^{\otimes m}$, $m \geq 0$, of the free group $\overline{\pi}$. See, for example, \S\ref{sec:exp} and \cite{Kaw05}. 
Since $\overline{\theta}_2(\Gamma_3\overline{\pi}) = 0$, we have 
$\overline{\theta}_2(\imath(\gamma)) = 0 \in \overline{H}^{\otimes 2}$. 
$\overline{\theta}_2\vert_{\Gamma_2\overline{\pi}}: \Gamma_2\overline{\pi} \to \overline{H}^{\otimes 2}$ is a group homomorphism. 
If we denote by $\overline{H}_{(1)} \subset \overline{H}$ the submodule generated by $\{[\overline{\alpha}_i], [\overline{\beta}_i]\}^g_{i=1}$, then we have 
\begin{equation}\label{eq:iotag}
\overline{\theta}_2(\imath(\gamma)) \equiv \sum^n_{j=1}\nu_j ([\overline{\alpha}_{g+j}]\otimes [\overline{\beta}_{g+j}] - [\overline{\beta}_{g+j}]\otimes [\overline{\alpha}_{g+j}]) \pmod{{\overline{H}_{(1)}}^{\otimes 2}}.
\end{equation}
In fact, the value of $\overline{\theta}_2\circ\imath$ at any commutator in $\{\alpha_i, \beta_i\}^g_{i=1}$ is in ${\overline{H}_{(1)}}^{\otimes 2}$. 
See the proof of Lemma \ref{lem:thze}. 
The element $\gamma'$ may include $\delta_j$'s as their commutators with some element $\xi$ of $\pi$, and $\overline{\theta}(\imath(\delta_j\xi{\delta_j}\inv\xi\inv)) \equiv 1+([\overline{\alpha}_{g+j}]\otimes [\overline{\beta}_{g+j}] - [\overline{\beta}_{g+j}]\otimes [\overline{\alpha}_{g+j}])\otimes [\xi] - [\xi]\otimes ([\overline{\alpha}_{g+j}]\otimes [\overline{\beta}_{g+j}] - [\overline{\beta}_{g+j}]\otimes [\overline{\alpha}_{g+j}]) \pmod{\overline{H}^{\otimes 4}}$. This proves \eqref{eq:iotag}. 
Hence we have $\nu_1=\nu_2 = \cdots = \nu_n = 0$, and so $\overline{\La}(\gamma) = 0$ for any $\gamma \in \Gwt_3\pi$, as was to be shown.
\end{proof}
\par
\bigskip
In the case where $P$ consists only of a single point $\ast \in \pa\Sigma$, the lift $\La$ has a non-commutative analogue $\tilde\La: \pi \to \pi_1(U\Sigma, v_\ast)$. By abuse of notation, we denote by $z \in \pi_1(U\Sigma, v_\ast)$ the fiber loop in $U\Sigma$. 
If we define 
\begin{equation}\label{eq:nc}
\tilde\La(\gamma) := \overset\cdot\gamma\cdot z^{\sharp\Gamma^+_\gamma - \sharp\Gamma^-_\gamma + 1} \in \pi_1(U\Sigma, v_\ast)
\end{equation}
for any generic representative $\gamma$ of an element of $\pi$, then it induces a well-defined map $\tilde\La: \pi \to \pi_1(U\Sigma, v_\ast)$ by a similar argument to $\Lambda$. Here it should be remarked that $z$ is central in the group $\pi_1(U\Sigma, v_\ast)$. \par

%
%

\section{$\bK$-framings, based $\bK$-framings and $\bK$-quadratic forms}\label{sec:frame}

Let $\bK$ be a commutative ring with unit. The characteristic 
$\ch\,\bK$ is the non-negative generator of the kernel of the ring homomorphism $\bZ \to\bK$ induced by the unit of $\bK$.
Let $\varpi: U\Sigma \to \Sigma$ be the unit tangent bundle of the surface $\Sigma$, as in the previous sections. For any $C^\infty$ immersion $\alpha: S^1= \bR/\bZ\to \Sigma$, we define 
$$
\overrightarrow{\alpha} := \overset\cdot\alpha_*[S^1] \in H_1(U\Sigma; \bZ), 
$$
using the velocity vector 
$\overset\cdot\alpha: S^1 \to U\Sigma = T^\times\Sigma/\bR_{>0}$, $t \mapsto (d\alpha)\left(\frac{\pa}{\pa t}\right)_t \bmod \bR_{>0}$, of the immersion $\alpha$, and denote
\begin{equation}\label{eq:rotw}
\rot_w(\alpha) := \langle w, \overrightarrow{\alpha}\rangle \in \bK.
\end{equation}
for any $w \in H^1(U\Sigma; \bK)$. 
Since the unit tangent bundle $\varpi$ is an oriented $S^1$-bundle, we have the Gysin exact sequence 
\begin{equation}\label{eq:ext}
0 \to  H^1(\Sigma; \bK) \overset{\varpi^*}\to H^1(U\Sigma; \bK) 
\overset{\iota^*}\to H^1(S^1; \bK) \overset{e}\to H^2(\Sigma; \bK).
\end{equation}
Here $\iota: S^1\hookrightarrow U\Sigma$ is the inclusion map of a fiber, and the map $e$ is given by $e(v) = \langle v, [S^1]\rangle\, e(\Sigma)$ for $v \in H^1(S^1; \bK)$, where $e(\Sigma) \in H^2(\Sigma; \bK)$ is the Euler class of the manifold $M$. 
Here we recall that $H^1(X; \bK) = \mathrm{Hom}_\bZ(H_1(X; \bZ), \bK)$ for any topological space $X$. \par

In the sequence \eqref{eq:ext}, we define a set $F(\Sigma; \bK) := \{\xi \in H^1(U\Sigma; \bK); \,\, \langle \iota^*(\xi), [S^1]\rangle = 1\}$ and call an element of the set $F(\Sigma; \bK)$ a {\it $\bK$-framing} of the surface $\Sigma$. 
By taking the composite with the second projection $\Sigma\times S^1 \to S^1$, the homotopy class of a usual framing $U\Sigma \cong \Sigma\times S^1$ is identified with a $\bZ$-framing $H_1(U\Sigma; \bZ) \to H_1(S^1; \bZ) = \bZ$. It can be regarded also as the homotopy class of a nowhere vanishing vector field on $\Sigma$. Moreover an $r$-spin structure is a fiberwise $\bZ/r$-covering of $U\Sigma$, and its corresponding homomorphism  $H_1(U\Sigma; \bZ) \to \bZ/r$ is a $\bZ/r$-framing. If $r=2$, a $2$-spin structure is exactly a spin structure.\par

From the sequence \eqref{eq:ext}, the set $F(\Sigma; \bK)$ is nonempty if and only if $e(\Sigma)=0\in H^2(\Sigma;\bK)$. 
It always holds if $n \geq 0$, i.e., the boundary $\pa\Sigma$ is nonempty. On the other hand, if $\Sigma$ is closed, i.e., $n=-1$, then it is equivalent to that $\ch\bK$ divides $\chi(\Sigma) = 2-2g$.\par

For the rest of this paper, we assume the set $F(\Sigma; \bK)$ is non-empty. The set $F(\Sigma; \bK)$ is an $H^1(\Sigma; \bK)$-torsor from the sequence \eqref{eq:ext}. So, for any $\xi, \xi' \in F(\Sigma; \bK)$ and  $u \in H^1(\Sigma; \bK)$, 
we simply write $\xi' = \xi + u$ or $\xi' - \xi = u$ for $\xi' = \xi + \varpi^*u$. \par

If $w = \xi \in F(\Sigma; \bK)$ is a $\bK$-framing, then we call the map $\rot_\xi$ in \eqref{eq:rotw} the {\it $\xi$-rotation number}. If $\xi' = \xi + u$, we have
\begin{equation}\label{eq:xixi'}
\rot_{\xi'}(\alpha) = \rot_\xi(\alpha) + \langle u, \alpha_*[S^1]\rangle.
\end{equation}
\begin{lem}[Poincar\'e-Hopf]\label{lem:PH}
We have 
\begin{equation}\label{eq:PH}
{\sum}^n_{j=0}\rot_\xi(\pa_j\Sigma) = \chi(\Sigma) \in \bK
\end{equation}
for any $\bK$-framing $\xi \in F(\Sigma; \bK)$. 
Here we endow each $\pa_j\Sigma$ with the orientation induced by that of the surface $\Sigma$.
\end{lem}
\begin{proof}
If $\Sigma$ is closed, then $\chi(\Sigma) = 0 \in \bK$ from the assumption $F(\Sigma; \bK) \neq \emptyset$. Hence the formula is trivial. 
If $\pa\Sigma \neq \emptyset$, there is a $\bK$-framing $\xi \in F(\Sigma; \bK)$ which lifts to an element of $\tilde\xi \in F(\Sigma; \bZ)$.
The formula for $\tilde\xi$ is exactly the classical Poincar\'e-Hopf theorem.
See, for example, \cite[Lemma 3.3]{Kaw18} for the proof. 
This implies the formula for $\xi$. Since $\sum^n_{j=0}[\pa_j\Sigma] =0 \in H_1(\Sigma; \bZ)$, we obtain the formula \eqref{eq:PH} for any $\xi' \in F(\Sigma; \bK)$ from \eqref{eq:xixi'}.
\end{proof}
At the end of this section, we will give an alternative proof of the Poincar\'e-Hopf formula \eqref{eq:PH} by using the quadratic form $q_\xi$ introduced in \eqref{eq:qxi}. 
\par
\medskip
Next we introduce the notion of a based $\bK$-framing.
Suppose $\pa\Sigma \neq \emptyset$. 
Let $P \subset \pa\Sigma$ and $\tilde P \subset U\Sigma\vert_{\pa\Sigma}$ be as in the previous section. 
Applying the snake lemma to the commutative diagram 
$$
\begin{CD}
0 @>>> \widetilde{H}^0(\tilde P; \bK)@>>> H^1(\Sigma, P; \bK) @>>> H^1(\Sigma; \bK)  @>>> 0\\
@. @| @V{\varpi^*}VV  @V{\varpi^*}VV @.\\
0 @>>> \widetilde{H}^0(\tilde P; \bK)  @>>> H^1(U\Sigma, \tilde P; \bK) @>{j^*}>> H^1(U\Sigma; \bK) @>>> 0, 
\end{CD}
$$
we obtain  
\begin{equation}\label{eq:extP}
0 \to H^1(\Sigma, P; \bK)\overset{\varpi^*}\to H^1(U\Sigma, \tilde P; \bK) \overset{\iota^*}\to H^1(S^1; \bK)  \to 0 \quad \text{(exact)}.
\end{equation}
Here we remark that the surjectivity of $\iota^*$ follows from that of $j^*$. 
In the sequence \eqref{eq:extP}, we define a set $F(\Sigma, P; \bK) := \{\xi \in H^1(U\Sigma, \tilde P; \bK); \,\, \langle \iota^*(\xi), [S^1]\rangle = 1\}$, and 
call an element of the set $F(\Sigma, P; \bK)$ a {\it based $\bK$-framing} of the pair $(\Sigma, P)$. The set $F(\Sigma, P; \bK)$ is non-empty and an $H^1(\Sigma, P; \bK)$-torsor. For a generic representative $\gamma: (I, \pa I) \to (\Sigma, P)$ for an element of $\Pi\Sigma\vert_P$, we have the $\xi$-rotation number $\rot_\xi(\gamma)$ by
$$
\rot_\xi(\gamma) := \langle \xi, \overrightarrow{\gamma}\rangle \in \bK.
$$
While $\rot_\xi(\gamma)$ depends on the representative $\gamma$, 
the quantity
\begin{equation}\label{eq:qxi}
q_\xi(\gamma) := (\xi\circ\La)(\gamma) = \langle \xi, \La(\gamma)\rangle = \rot_\xi(\gamma) + \sharp\Gamma^+_\gamma - \sharp\Gamma^-_\gamma + 1 \in \bK
\end{equation}
is independent from the choice of a generic representative of the morphism $\gamma$ in $\Pi\Sigma\vert_P$. 
The right hand side of \eqref{eq:qxi} can be seen as the writhe of $\gamma$ with respect to the $\bK$-framing $\xi$. 
\par
The map $q_\xi: \pi = \Pi\Sigma(\ast, \ast) \to \bK$ is closely related to Turaev's self-intersection based loop operation \cite{Tu1979}. In fact, we have 
\begin{equation}
q_\xi = (\mathrm{aug}\otimes\mathrm{aug})\circ \mu^\xi_r: \pi \to |\bK\pi|\otimes \bK\pi \to \bK,
\end{equation}
where $\mu^\xi_r$ is a framed analogue \cite[\S2.3]{AKKNb} of Turaev's operation, 
$|\bK\pi|$ is the trace space of the group algebra $\bK\pi$, 
and $\mathrm{aug}: \bK\pi \to \bK$ and $\mathrm{aug}: |\bK\pi| \to \bK$ are the augmentation maps. 
Here the initial and terminal vectors $\overset\cdot{\gamma}(\epsilon) \in T_\ast\Sigma$, $\epsilon = 0,1$, of the immersion $\gamma$ are perpendicular to the boundary $\pa\Sigma$, so that we need to add $\frac12$ to the rotation number in \cite{AKKNb} to get that in this paper. 
The map $q_\xi$ on the whole groupoid $\Pi\Sigma\vert_P$ also should be related to a groupoid version of $\mu^\xi_r$ intruduced by Taniguchi \cite{Tan25b}. \par
The product formula for $\La$ in Lemma \ref{lem:product} implies
\begin{equation}\label{eq:qprod}
q_\xi(\gamma_1\gamma_2) = q_\xi(\gamma_1) + q_\xi(\gamma_2) + \gamma_1\cdot\gamma_2 \in \bK
\end{equation}
for any composable morphisms $\gamma_1$, $\gamma_2$ in $\Pi\Sigma\vert_P$. 
Following Johnson \cite{J80a}, we introduce the notion of a $\bK$-quadratic form on $(\Sigma, P)$. 
Namely, a map $q: \Mor(\Pi\Sigma\vert_P) \to \bK$ from the set $\Mor(\Pi\Sigma\vert_P)$ of all morphisms in the groupoid $\Pi\Sigma\vert_P$ is a {\it $\bK$-quadratic form} on $(\Sigma, P)$ if 
$$
q(\gamma_1\gamma_2) = q(\gamma_1) + q(\gamma_2) + \gamma_1\cdot \gamma_2 \in \bK
$$
for any composable $\gamma_1, \gamma_2 \in \Mor(\Pi\Sigma\vert_P)$. 
We denote by $Q(\Sigma, P; \bK)$ the set of all $\bK$-quadratic form on $(\Sigma, P)$. 
It is non-empty, since $q_\xi \in Q(\Sigma, P; \bK)$ for any $\xi \in F(\Sigma, P; \bK)$. 
For any $u \in H^1(\Sigma, P; \bK)$, the map $q+u: \Mor(\Pi\Sigma\vert_P) \to \bK$, 
$\gamma \mapsto q+u(\gamma) := q(\gamma) + \langle u,  \gamma_*[I, \pa I]\rangle \in \bK$ is also a $\bK$-quadratic form on $(\Sigma, P)$. 
Here $[I, \pa I] \in H_1(I, \pa I; \bZ)$ is the fundamental class of the unit interval $I = [0,1]$ as before. 
\begin{prop}\label{prop:FQ} The set $Q(\Sigma, P; \bK)$ is an $H^1(\Sigma, P; \bK)$-torsor, and the map 
$$
\mathfrak{q}: F(\Sigma, P; \bK) \to Q(\Sigma, P; \bK), \quad \xi \mapsto \mathfrak{q}(\xi) := q_\xi,
$$
is an isomorphism of $H^1(\Sigma, P; \bK)$-torsors.
\end{prop}
\begin{proof} In this proof, we regard any abelian group as a groupoid with a single object. For any $q, q' \in Q(\Sigma, P; \bK)$, the difference $q'-q: \Mor(\Pi\Sigma\vert_P) \to \bK$ is a functor. \par
The evaluation map at the fundamental class 
$[I, \pa I]$ defines the abelianization functor $\mathrm{ab}: \Pi\Sigma\vert_P \to H_1(\Sigma, P; \bZ)$. For any functor $F: \Pi\Sigma\vert_P \to A$ to any abelian group $A$, we have a unique group homomorphism $\bar F: H_1(\Sigma, P; \bZ) \to A$ such that $F = \bar F\circ\mathrm{ab}: \Pi\Sigma\vert_P \to A$. We can prove it in a similar way to the fact that the abelianization of the fundamental group of a path-connected space is canonically isomorphic to the integral first homology group. 
See, for example, \cite[pp.166-167]{Hat02}. 
Hence $q'-q$ induces a group homomorphism $H_1(\Sigma, P; \bZ) \to \bK$, or equivalently an element of $H^1(\Sigma, P; \bK)$. 
Therefore $Q(\Sigma, P; \bK)$ is an $H^1(\Sigma, P; \bK) = \Hom_{\bZ}(H_1(\Sigma, P; \bZ), \bK)$-torsor. From the definitions of $q+u$ and $\xi+u$, we have $q_{\xi+u} = q_\xi + u$. This proves the proposition.
\end{proof}
It is clear that the map $\mathfrak{q}$ is the mapping class group equivariant.
The proposition enables us to regard any element of $Q(\Sigma, P; \bK)$ as a based $\bK$-framing of the pair $(\Sigma, P)$. 
\par
\bigskip
Now we consider the case $P = \{\ast\}$. 
Then we have the extension \eqref{eq:extP} equals the extension \eqref{eq:ext}, 
and so $F(\Sigma, P; \bK) = F(\Sigma; \bK)$. Proposition \ref{prop:FQ} says that 
the map 
$$
\mathfrak{q}: F(\Sigma; \bK) \to Q(\Sigma, \{\ast\}; \bK), \quad \xi \mapsto \mathfrak{q}(\xi) = q_\xi,
$$
is an isomorphism of $H^1(\Sigma; \bK)$-torsors. From Theorem \ref{thm:wtG}, we may identify
$$
Q(\Sigma, \{\ast\}; \bK) = \{q: \pi/\Gwt_3\pi \to \bK; \,\, \forall \gamma_1, \forall \gamma_2 \in \pi, \,\, q(\gamma_1\gamma_2) = q(\gamma_1) + q(\gamma_2) + \gamma_1\cdot\gamma_2\}.
$$
By the same computation as Lemma \ref{lem:commutator} we have
\begin{equation}\label{eq:commuator}
q(\gamma_1\gamma_2{\gamma_1}\inv{\gamma_2}\inv) = 2\gamma_1\cdot\gamma_2 
\in \bK.
\end{equation}
In particular if $\bK = \bZ/2$, we may identify $Q(\Sigma, \{\ast\}; \bZ/2)$ with the set 
$$
Q_2(\Sigma) := \{\bar q: H_1(\Sigma; \bZ/2) \to \bZ/2; \,\, \forall x_1, \forall x_2 \in \pi, \,\, \bar q(x_1 + x_2) = \bar q(x_1) + \bar q(x_2) + x_1\cdot x_2\}.
$$
Proposition \ref{prop:FQ} implies the isomorphism 
\begin{equation}\label{eq:spinS}
\mathfrak{q}: F(\Sigma; \bZ /2) \overset\cong\to Q_2(\Sigma), \quad \xi \mapsto q_\xi,
\end{equation} 
of $H^1(\Sigma; \bZ/2)$-torsors. \par
If $\Sigma = \Sigma_{g,1}$, then the inclusion $i: \Sigma = \Sigma_{g,1} \hookrightarrow \Sigma_{g,0}$ induces an isomorphism $H_1(\Sigma_{g,1}; \bZ/2) \cong H_1(\Sigma_{g,0}; \bZ/2)$. 
Hence we have $Q_2(\Sigma_{g,1}) = Q_2(\Sigma_{g,0})$. The latter is the set of theta characteristics of genus $g$ in the classical context. 
Moreover we have $H_*(U\Sigma_{g,0}, U\Sigma_{g,1}; \bZ) \underset{\text{exc}}\cong H_*(D^2\times S^1, S^1\times S^1; \bZ) \cong H_*(D^2, S^1; \bZ) \otimes H_*(S^1; \bZ) \cong H_{*-2}(S^1; \bZ)$. 
Hence $H_1(U\Sigma_{g,1}; \bZ) \cong H_1(U\Sigma_{g,0}; \bZ)$ and so 
$F(\Sigma_{g,1}; \bZ/2) \cong F(\Sigma_{g}; \bZ/2)$. Consequently we obtain an isomorphism 
\begin{equation}\label{eq:Johnson}
\mathfrak{q}: F(\Sigma_{g,0}; \bZ/2) \overset\cong\to  Q_2(\Sigma_{g,0}), \quad \xi \mapsto q_\xi,
\end{equation}
of $H^1(\Sigma_{g,0}; \bZ/2)$-torsors. 
The isomorphism \eqref{eq:Johnson} was established by Johnson \cite[Theorem 3A, p.371]{J80a}. \par

Morita \cite[\S6, p.803]{Mo89} constructed an explicit element $d \in Q(\Sigma_{g,1}, \{\ast\}; \bZ)$ based on an explicit free symplectic generator system $\{\alpha_i, \beta_i\}^g_{i=1}$ of the fundamental group $\pi = \pi(\Sigma_{g,1})$ depicted in \cite[Figure 3, p.801]{Mo89}. 
He uses it to construct an explicit cocycle representing the Earle class $k_{\bZ} \in H^1(\Gamma_{g,1}; H^1(\Sigma_{g,1}; \bZ))$ which will be explained in \S\ref{sec:k}. 
We can take generic representatives for $\alpha_i, \beta_i$ in our convention in \S\ref{sec:lift} such that $\alpha_i$ is simple and $\beta_i$ has one negative self-intersection point. 
Then, if $\xi_d \in F(\Sigma_{g,1}; \bZ)$ is the $\bZ$-framing with $q_{\xi_d} = d$, these representatives satisfy $\rot_{\xi_d}(\alpha_i) = d(\alpha_i) - 1 = -1$ and $\rot_{\xi_d}(\beta_i) = d(\beta_i) - 2 = -2$. \par

We conclude this section by an alternative proof of the Poincar\'e-Hopf formula \eqref{eq:PH}. Let $\xi \in F(\Sigma; \bK)$ be a $\bK$-framing on $\Sigma = \Sigma_{g,n+1}$. Similar to the proof of Lemma \ref{lem:PH}, we may assume $n \geq 0$. Recall our basepoint $\ast$ is located on the $0$-th boundary component $\pa_0\Sigma$.
For each element in $\pi = \pi_1(\Sigma, *)$, we take a representative $\gamma: I = [0,1] \to \Sigma$ with negative initial vectors $\overset\cdot\gamma(0) = \overset\cdot\gamma(1) \in T_*\pa_0\Sigma$. 
Now we have a standard free generator system $\{\alpha_i, \beta_i\}^g_{i=1}\cup \{\delta_j\}^n_{j=1}$ of the fundamental group $\pi$ such that 
$$
\zeta_0 := \delta_1\delta_2\cdots \delta_n\prod^g_{i=1}\alpha_i\beta_i{\alpha_i}\inv {\beta_i}\inv \in\pi
$$
is represented by the negative boundary loop $(\pa_0\Sigma)\inv$, and $\delta_j$ is represented by an embedded loop which is freely regular homotopic to the boundary loop $\pa_j\Sigma$ for each $j \neq 0$. 
Since $(\pa_0\Sigma)\inv$ is an embedded loop, we have $q_\xi(\zeta_0) = \rot_\xi((\pa_0\Sigma)\inv) + 1 = -\rot_\xi(\pa_0\Sigma) + 1$. 
For each $j \neq 0$, we have $q_\xi(\delta_j) = \rot_\xi(\pa_j\Sigma) + 1$. 
Hence, from the product formula \eqref{eq:qprod}, we have 
$$
-\rot_\xi(\pa_0\Sigma) + 1 = q_\xi(\zeta_0) 
= 2g + n + \sum^n_{j=1}\rot_\xi(\pa_j\Sigma),
$$
which means the Poincar\'e-Hopf formula \eqref{eq:PH}: $\sum^n_{j=0}\rot_\xi(\pa_j\Sigma) = 1- 2g - n = \chi(\Sigma_{g,n+1})$.
\par

\section{The $\bK$-Earle class $k_\bK$}\label{sec:k}

Let $\Gamma = \Gamma_{g,n+1}$ be the mapping class group of the surface $\Sigma = \Sigma_{g,n+1}$ fixing the boundary {\it pointwise}. The group $\Gamma$ acts on the (co)homology groups of $\Sigma$, $(\Sigma, P)$, $(\Sigma, \pa\Sigma)$, 
$U\Sigma$, $(U\Sigma, \tilde P)$ and so on in an obvious way.
In particular, we have $\varphi\cdot\xi = \xi\circ (d\varphi\inv)_* \in H^1(U\Sigma; \bK)$ for any $\xi \in F(\Sigma; \bK)$ and $\varphi \in \Gamma$. 
We remark that the Poincar\'e duality is $\Gamma$-equivariant. \par
Also in this section we suppose $F(\Sigma; \bK) \neq \emptyset$. 
Then the sequence \eqref{eq:ext} 
$$
0 \to H^1(\Sigma; \bK) \to H^1(U\Sigma; \bK) \to \bK \to 0
$$
is an extension of $\bK\Gamma$-modules. 
Furuta's interpretation \cite[pp.569-]{Mo97} defines a cohomology class $k_\bK \in H^1(\Gamma; H^1(\Sigma; \bK))$ by the extension class of the extension. In other words, $k_\bK$ is the cohomology class of the cocycle $k_\xi$ defined by 
\begin{equation}\label{kxi}
k_\xi(\varphi) := \varphi\cdot\xi - \xi \in H^1(\Sigma; \bK) 
\end{equation}
for $\xi \in F(\Sigma; \bK)$ and $\varphi \in \Gamma$.
Here we have $q_{\xi}(\varphi\inv(\gamma))-q_\xi(\gamma) = (q_{\varphi\cdot\xi}-q_\xi)(\gamma) = (\varphi\cdot\xi - \xi) \circ \La(\gamma) = k_\xi(\varphi)(\La(\gamma)) = k_\xi(\varphi)[\gamma]$, i.e., 
\begin{equation}\label{eq:kqq}
k_\xi(\varphi) = q_{\xi}\circ\varphi\inv-q_\xi: \pi \to \bK.
\end{equation}
\par
Morita \cite{Mo89} computed $H^1(\Gamma_{g,1}; H_\bZ) = \bZ$ for any $g \geq 2$, where the element $d \in Q(\Sigma_{g,1}; \bZ)$ stated in \S\ref{sec:frame} induces a generator of the cohomology group. Moreover he proved that the induced homomorphism of the group homomorphism $\pi_1(U\Sgone) \to \Gamma_{g,1}$ defined by push-maps maps its generator to $\pm(2-2g)$ times the Poincar\'e duality isomorphism $H_1(\pi_1(U\Sigma_{g,1}); \bZ) = H_1(\Sigma_{g,1}; \bZ) \to H^1(\Sigma_{g,1}; \bZ)$. 
Prior to Morita's results \cite{Mo89}, Earle \cite{E78} gave a cocycle representing $k_\bZ$ in the context of the Abel-Jacobi map. 
Morita \cite[Proposition 4.1]{Mo97} proved that the extension class $k_\bZ$ generates the group $H^1(\Gamma_{g,1}; H_\bZ)$ by evaluating it on $\pi_1(U\Sgone)$. 
For any simple closed curve $C$ and any $\xi \in F(\Sigma_{g,1}; \bZ)$, 
the value of $k_\xi$ at the push-map along $C$ equals $2g-2$ times the Poincar\'e duality of the homology class $[C] \in H_1(\Sigma_{g,1}; \bZ)$. 
For the evaluation, see also \cite[Theorem 1.3]{KS1}, where we made a sign mistake in the formula for $\rot_\xi(T_C(\alpha)) - \rot_\xi(\alpha)$. 
On the other hand, Kuno \cite{Ku09} gave a combinatorial formula computing the cocycle defined by Earle \cite{E78}. Moreover Kuno, Penner and Turaev \cite{KPT} introduced an explicit cocycle representing a nonzero multiple of $k_\bZ$ on the Ptolemy groupoid. For the nonzero multiple, see \cite[Theorem 2.5]{Ku17}.
\par
For any compact connected oriented surface $\Sigma$, we have
\begin{lem}\label{lem:DT}
Let $C \subset \Sigma$ be a simple closed curve. Then the value of $k_\xi$ at the right-handed Dehn twist $T_C$ along $C$ is given by
$$
k_\xi(T_C) = (\rot_\xi C)(\cdot [C]) \in H^1(\Sigma; \bK),
$$
where $\cdot [C]\in H^1(\Sigma, \pa \Sigma; \bZ)$ means the algebraic intersection with the homology class $[C] \in H_1(\Sigma; \bZ)$, i.e., the Poincar\'e duality of the class $[C]$. 
\end{lem}
\begin{proof} It suffices to check the formula by evaluating at any immersed loop $\alpha: S^1 \to \Sigma$. As before, we denote by $\overrightarrow{\alpha} \in H_1(U\Sigma; \bZ)$ the homology class of the velocity vector $\overset\cdot\alpha: S^1 \to U\Sigma$. Then we have 
$$
\aligned
& \langle k_\xi(T_C), [\alpha]\rangle
= \langle T_C\cdot \xi - \xi, [\alpha]\rangle
= \langle \xi \circ (d{T_C}\inv)_*, \overrightarrow{\alpha}\rangle -\langle \xi, \overrightarrow{\alpha}\rangle\\
&= \rot_\xi({T_C}\inv \alpha) - \rot_\xi(\alpha) =  ([\alpha]\cdot[C])\rot_\xi(C).
\endaligned
$$
This proves the lemma. 
\end{proof}
Now we can determine the (non)triviality of the class $k_{\bK}$.
\begin{thm}\label{thm:nontrivial} The class $k_\bK \in H^1(\Gamma_{g,n+1}; H^1(\Sigma_{g,n+1}; \bK))$ is nontrivial if and only if $g \geq 2$.
\end{thm} 
\begin{proof} For any simple closed curve $C \subset \Sigma$ and any $u \in H^1(\Sigma; \bK)$, the value $(\delta u)(T_C)$ of the coboundary $\delta u \in Z^1(\Gamma; H^1(\Sigma; \bK))$ at the Dehn twist $T_C$ is given by
\begin{equation}\label{eq:cob}
(\delta u)(T_C) = \langle u, [C]\rangle (\cdot [C]).
\end{equation}
In fact, for any loop $\gamma$ in $\Sigma$, 
we have $
\langle (\delta u)(T_C), [\gamma]\rangle =
\langle u\circ ({T_C}\inv) - u, [\gamma]\rangle =
\langle u, ({T_C}\inv)_*[\gamma] -[\gamma]\rangle=
\langle u, ([\gamma] \cdot [C])[C]\rangle=
\langle u, [C]\rangle ([\gamma] \cdot [C])
$.
\par
(I) The case $g=0$. If $n \leq 0$, it is trivial since $H^1(\Sigma_{0,n+1}; \bK) = 0$. So we suppose $n \geq 1$. Then the homology classes $\overrightarrow{\pa_j\Sigma}$, $1 \leq j \leq n$, and the fiber class $z$ generate
the homology group $H_1(U\Sigma; \bZ)$. The mapping class group $\Gamma = \Gamma_{0,n+1}$ acts trivially on these homology classes, and so on the cohomology group $H^1(U\Sigma; \bK)$. Hence, in the case $g =0$, we have 
\begin{equation}\label{eq:g0kxi}
k_\xi = 0: \Gamma \to H^1(\Sigma; \bK).
\end{equation}
\par
(II) The case $g=1$. By capping disks on all the boundary components, we can embed $\Sigma_{1,n+1}$ into the torus $\Sigma_{1,0}$. This induces a map $H^1(\Gamma_{1,0}; H^1(\Sigma_{1,0}; \bK)) \to H^1(\Gamma_{1,n+1}; H^1(\Sigma_{1,n+1}; \bK))$ which maps $k_{\bK}$ for $\Sigma_{1,0}$ to 
that for $\Sigma_{1,n+1}$, since any $\bK$-framing on $\Sigma_{1,0}$ induces a $\bK$-framing on $\Sigma_{1,n+1}$.  Here it should be remarked the torus is parallelizable. \par
Hence it suffices to show the vanishing of $k_{\bK}$ for $\Sigma_{1,0}$. 
Under the identification $\Sigma_{1,0} = \bR^2/\bZ^2$, we take two emmbeded loop $\alpha: t \in I \mapsto (t,0)\bmod{\bZ^2} \in \Sigma_{1,0}$ and $\beta: t \in I \mapsto (0, t)\bmod{\bZ^2} \in \Sigma_{1,0}$. 
Let $\xi \in F(\Sigma_{1,0}; \bK)$. 
From \eqref{eq:cob} we have $(\delta(\cdot[\gamma]))(T_C) = -([\gamma]\cdot[C])(\cdot [C])$ for any loop $\gamma$ in $\Sigma$. 
Hence we compute the coboundary $\delta u_\xi$ of the element $u_\xi := -  (\rot_\xi\beta)(\cdot [\alpha]) + (\rot_\xi\alpha)(\cdot [\beta]) \in H^1(\Sigma_{1,0}; \bK)$ as
$$
(\delta u_\xi)(T_C) = \left((\rot_\xi\beta)([\alpha]\cdot [C]) - (\rot_\xi\alpha)([\beta]\cdot [C]) \right)(\cdot [C]).
$$
In particular, we have $(\delta u_\xi)(T_\alpha) = (\rot_\xi\alpha)(\cdot [\alpha]) = k_\xi(T_\alpha)$ and $(\delta u_\xi)(T_\beta) = (\rot_\xi\beta)(\cdot [\beta]) = k_\xi(T_\beta)$ from Lemma \ref{lem:DT}. 
This implies $\delta u_\xi = k_\xi: \Gamma_{1,0} \to H^1(\Sigma_{1,0}; \bK)$,  
since the mapping class group $\Gamma = \Gamma_{1,0} \cong SL_2(\bZ)$ is generated by the Dehn twists $T_\alpha$ and $T_\beta$. This proves $k_\bK = 0$ in the case $g=1$. \par
(III) The case $g \geq 2$. Then we have a subsurface $S$ of $\Sigma$ which is diffeomorphic to a pair of pants, and each of whose boundary components is non-separating in the surface $\Sigma$. Let $C_k := \pa_kS$, $1 \leq k \leq 3$, be the three boundary components of the pair of pants $S$. 
Now assume that $k_\xi = \delta u$ for some $u \in H^1(\Sigma; \bK)$. 
From Lemma \ref{lem:DT} and \eqref{eq:cob}, we have $(\rot_\xi C_k)(\cdot [C_k])
= k_\xi(T_{C_k}) = (\delta u)(T_{C_k}) = \langle u, [C_k]\rangle (\cdot [C_k])$. 
Since $C_k$ is non-separating, we have $[\gamma_k]\cdot [C_k] = 1$ for some loop $\gamma_k$ in $\Sigma$. Hence we obtain $\rot_\xi C_k = \langle u, [C_k]\rangle$.
Consequently we obtain a contradiction 
$$
-1 = \chi(S) = {\sum}^3_{k=1}\rot_\xi C_k = \langle u, {\sum}^3_{k=1}[C_k]\rangle = 0
$$
from Lemma \ref{lem:PH} and $\sum^3_{k=1}[C_k] = 0 \in H_1(\Sigma; \bZ)$. 
This implies there is no $u \in H^1(\Sigma; \bK)$ satisfying $k_\xi = \delta u$, and 
proves $k_\bK \not= 0$ in the case $g\geq 2$. 
\end{proof}

The cohomology class $k_\bK$ can be regarded as the set of $1$-cocycles representing itself: $k_\bK \subset Z^1(\Gamma; H^1(\Sigma; \bK))$. 
Here $Z^1(\Gamma; M)$ means the module of $1$-cocycles of the group $\Gamma$ with coefficients in a $\Gamma$-module $M$. 
As will be shown, it can be identified with $F(\Sigma; \bK)$ as $H^1(\Sigma;\bK)$-torsors if $g \geq 1$.
In order to prove it, we consider the homology exact sequence of the pair
$(\Sigma, \pa\Sigma)$
\begin{equation}\label{eq:pair}
\cdots \to H_2(\Sigma, \pa\Sigma; \bZ) \overset{\pa_*}\to 
H_1(\pa\Sigma; \bZ) \overset{i_*}\to 
H_1(\Sigma; \bZ) \overset{j_*}\to 
H_1(\Sigma, \pa\Sigma; \bZ) \overset{\pa_*}\to 
\tilde H_0(\pa\Sigma; \bZ) \to 0,
\end{equation}
on which the mapping class group $\Gamma$ acts equivariantly.
\begin{lem}\label{lem:coinvariants}
{\rm (1)} The coinvariants $H_1(\Sigma; \bZ)_\Gamma$ vanishes if $g \geq 1$, while it equals $H_1(\Sigma; \bZ) = i_*(H_1(\pa\Sigma; \bZ))$ if $g =0$.
It is $\bZ$-free in any cases. \par
{\rm (2)} The map $\pa_*$ induces an isomorphism of 
the coinvariants $H_1(\Sigma, \pa\Sigma; \bZ)_\Gamma$
 onto $\tilde H_0(\pa\Sigma; \bZ)$ for any $g \geq 0$.\par
 {\rm (3)} The invariants $H^1(\Sigma; \bK)^\Gamma$ vanishes if $g \geq 1$, while it is isomorphic to the kernel $\Ker(\delta^*: H^1(\pa\Sigma; \bK) \to H^2(\Sigma, \pa\Sigma; \bK))$ 
 if $g =0$.\par
 {\rm (4)} The invariants $H^1(\Sigma, \pa\Sigma; \bK)^\Gamma$ is isomorphic to $\tilde H^0(\pa\Sigma; \bK)$ 
 through the connecting homomorphism $\delta^*$ for any $g \geq 0$. 
\end{lem}
\begin{proof} (1) Suppose $g \geq 1$. Then there exists a non-separating simple closed curve on $\Sigma$, and the homology group $H_1(\Sigma; \bZ)$ is generated by non-separating simple closed curves and the boundary curves.
For any  non-separating simple closed curve $\alpha$, there is a  non-separating simple closed curve $\beta$ with $\alpha\cdot\beta = 1$. The right-handed Dehn twist $T_\alpha$ along $\alpha$ satisfies $T_\alpha[\beta] - [\beta] = [\alpha] \in H_1(\Sigma; \bZ)$, so that $[\alpha]$ vanishes in the coinvariants $H_1(\Sigma; \bZ)_\Gamma$. On the other hand, for any $0 \leq j \leq n$, we can take non-separating simple closed curves $\alpha$, $\beta$ and $\gamma$ as in the figure.
\begin{figure}[h]
\begin{center}
\begin{tikzpicture}
\coordinate (U) at (0,2);
\coordinate (D) at (0,0);
\draw[very thick, bend right = 15] (U) to (D);
\node at (-0.167,1) {$\wedge$};
\node at (0.7,1) {$\partial_j\Sigma$};
\draw[very thick, bend left = 15] (U) to (D);
\draw[very thick] (U) to [out=180, in=330] (-5,2.5);
\draw[very thick] (D) to [out=180, in=30] (-5,-0.5);
\draw[very thick] (-2,1) circle (0.4);
\draw[thick] (-2,1) circle (0.7);
\node at (-2.7,1) {$\vee$};
\node at (-3.0,1) {$\beta$};
\coordinate (ud) at (-2,1.4);
\coordinate (du) at (-2,0.6);
\coordinate (uu) at (-2,1.9);
\coordinate (dd) at (-2,0.1);
\draw[thick, bend right = 15] (ud) to (uu);
\draw[thick, bend left = 15] (dd) to (du);
\node at (-2.03,0.3) {$\vee$};
\node at (-2,-0.3) {$\gamma$};
\node at (-1.97,1.7) {$\wedge$};
\node at (-2,2.2) {$\alpha$};
\end{tikzpicture}
\end{center}
\label{fig:abc}
\end{figure}
Then we have $T_\gamma[\beta] - [\beta] = [\gamma] = [\alpha] - [\pa_j\Sigma]$, so that $[\pa_j\Sigma]$ also vanishes in the coinvariants. Hence we obtain $H_1(\Sigma; \bZ)_\Gamma = 0$ if $g \geq 1$.\par
If $g = 0$, the inclusion homomorphism $i_*: H_1(\pa\Sigma; \bZ) \to H_1(\Sigma; \bZ)$ is surjective, and $\Gamma$ acts on $H_1(\pa\Sigma; \bZ)$ trivially. Hence $\Gamma$ acts on $H_1(\Sigma; \bZ)$ trivially.\par
(2) By the right exactness of the tensor product $\otimes_{\bZ\Gamma}\bZ$,
the exact sequence \eqref{eq:pair} induces an exact sequence
$$
H_1(\Sigma; \bZ)_\Gamma \overset{j_*}\to H_1(\Sigma, \pa \Sigma; \bZ)_\Gamma \overset{\pa_*}\to \tilde H_0(\pa\Sigma; \bZ)_\Gamma \to 0,
$$
where we have $j_* = 0$ for any $g \geq 0$. This implies (2).
In fact, for $g \geq 1$, we have $H_1(\Sigma; \bZ)_\Gamma = 0$  from (1). For $g =0$, it follows from $j_*i_* = 0: H_1(\pa\Sigma; \bZ) \to H_1(\Sigma, \pa \Sigma; \bZ)$.  \par
(3) and (4) follow from (1) and (2), respectively.
\end{proof}

From Lemma \ref{lem:coinvariants} (3) together with the exact sequence 
$$
0 \to H^1(\Sigma; \bK)^\Gamma \to H^1(\Sigma; \bK) \overset{\delta}\to 
Z^1(\Gamma; H^1(\Sigma; \bK)) \to H^1(\Gamma; H^1(\Sigma; \bK)) \to 0,
$$
we obtain 
\begin{cor}\label{cor:kK} If $g \geq 1$, then the map
$$
\xi \in F(\Sigma; \bK) \mapsto k_\xi \in k_{\bK} \left(\subset Z^1(\Gamma; H^1(\Sigma; \bK)) \right),
$$
is a bijection.
\end{cor}
In fact, if $g \geq 1$, the set $k_{\bK}$ is a $H^1(\Sigma; \bK)$-torsor, and the map is $H^1(\Sigma; \bK)$-equivariant. \par
Consequently, if $g \geq 1$, we have isomorphisms of $H^1(\Sigma; \bK)$-torsors
$$
Q(\Sigma; \bK) \cong F(\Sigma; \bK) \cong k_{\bK}.
$$
On the other hand, if $g=0$, the map $\xi \in F(\Sigma_{0,n+1}; \bK) \to k_\xi = 0 \in k_\bK = \{0\}$ is constant by \eqref{eq:g0kxi} and Lemma \ref{lem:coinvariants} (1). \par

Now we assume $\pa\Sigma\neq \emptyset$, and consider the finite subsets $P \subset \pa\Sigma$ and $\tilde P \subset U\Sigma$ discussed in \S\ref{sec:lift} and \S\ref{sec:frame}. 
Then we have the extension class $k^\natural_\bK \in H^1(\Gamma; H^1(\Sigma, P;\bK))$ of the extension \eqref{eq:extP} of $\Gamma$-modules. 
For any $\xi \in F(\Sigma, P; \bK)$, it is represented by the cocycle $k_\xi: \Gamma \to H^1(\Sigma, P;\bK)$ defined by 
$$
k_\xi(\varphi) = \varphi\cdot \xi - \xi \in H^1(\Sigma, P;\bK)
$$
for $\varphi \in \Gamma$. 
The inclusion $U\Sigma \hookrightarrow (U\Sigma, \tilde P)$ induces a map $F(\Sigma, P; \bK) \to F(\Sigma; \bK)$ which is equivariant under the inclusion homomorphism $j'^*_P: H^1(\Sigma, P;\bK) \to H^1(\Sigma; \bK)$. 
Since all these are $\Gamma$-equivariant, we have 
\begin{equation}\label{eq:natlift}
(j'^*_P)_*k^\natural_\bK = k_\bK \in H^1(\Gamma; H^1(\Sigma; \bK)).
\end{equation}
Thus the cohomology class $k^\natural_\bK$ is a lift of $k_\bK$. 
We will discuss a lift of $k^\natural_\bK$ in the cohomology group $H^1(\Gamma; H^1(\Sigma, \pa\Sigma; \bK))$ in \S\ref{sec:relf}. 
\par
In order to study the moduli space of twisted holomorphic $1$-forms on Riemann surfaces, Apisa and Salter \cite[Theorem A.4]{ApisaSalter} computed $H^1(\Gamma; H^1(\Sigma, \pa\Sigma; \bK))$ and proved a natural isomorphism $F(\Sigma, P; \bK) \cong k^\natural_\bK$ for $\Sigma = \Sigma_{g,n+1}$ with $g \geq 3$ and $n+1 \geq 3$. 
Their computation of the first cohomology group $\Gamma$ is quite neat. 
\par

%
%

\section{$\bK$-expansions}\label{sec:exp}

In this section, we suppose $\pa\Sigma \neq\emptyset$, and choose a basepoint $\ast$ on the boundary $\pa\Sigma$. 
Then the fundamental group $\pi := \pi_1(\Sigma, \ast)$ is a free group of finite rank, so that we can consider its $\bK$-expansions stated below. 
Let $H = \HK := H_1(\Sigma; \bK)$ denote the first homology group of the surface $\Sigma$, $[\gamma] \in \HK$ the homology class of an element $\gamma \in \pi$. 
A map $\theta = \sum^\infty_{m=0}\theta_m: \pi \to \widehat{T}(\HK) := \prod^\infty_{m=0}{\HK}^{\otimes m}$, $\theta_m: \pi \to {\HK}^{\otimes m}$, $m \geq 0$, is a (generalized Magnus) {\it $\bK$-expansion}, if it is a group homomorphism into the multiplicative group of the completed tensor algebra $\widehat{T}(\HK)$, $\theta_0(\gamma) = 1$ and $\theta(\gamma) = [\gamma] \in \HK$ for any $\gamma \in \pi$. See, for example, \cite{Kaw05}. 
Since $\pi$ is a free group, expansions do exist. 
In particular, we may choose any element of $H^{\otimes m}$ as $\theta_m(\gamma_i)$ for any $m \geq 2$, if $\{\gamma_i\}_{i=1}^{b_1(\Sigma)}$ is a free generating system of the free group $\pi$. \par
Now we can construct a $\bK$-quadratic form from any $\bK$-expansion. 
Let $\flat: H_1(\Sigma; \bK)^{\otimes 2} \to \bK$ be the algebraic intersection form, i.e., $\flat(X_1X_2) = X_1\cdot X_2$ for $X_1, X_2 \in H$. 
Here and for the rest of this section, we often write $X_1X_2 = X_1\otimes X_2 \in H^{\otimes 2}$ and similarly $X_1X_2X_3 = X_1\otimes X_2\otimes X_3 \in H^{\otimes 3}$ for any $X_1, X_2, X_3 \in H$. 
\begin{lem}\label{lem:C}
The composite $\flat\circ \theta_2: \pi \to \bK$ of the second term $\theta_2$ of any expansion $\theta$ and the form $\flat$ is a $\bK$-quadratic form: 
$\flat\circ\theta_2 \in Q(\Sigma, \{\ast\}; \bK)$. 
\end{lem}
\begin{proof} Since $\theta$ is a group homorphism, we have 
$\theta_2(\gamma_1\gamma_2) = \theta_2(\gamma_1) + \theta_2(\gamma_2) + [\gamma_1]\otimes [\gamma_2] \in H_1(\Sigma; \bK)^{\otimes 2}$ for any $\gamma_1$ and $\gamma_2 \in \pi$. Applying the form $\flat$ to the equation, we obtain the lemma. 
\end{proof}

If the genus $g$ is positive, there is a unique $\bK$-framing $\xi_\theta \in F(\Sigma; \bK)$ satisfying $\mathfrak{q}(\xi_\theta) = \flat\circ\theta_2$ from Corollary \ref{cor:kK}. Here we remark that $k_\xi = 0$ for any $\xi \in F(\Sigma; \bK)$ if $g=0$, from the proof of Theorem \ref{thm:nontrivial}. So we suppose $g \geq 1$ for the rest of this section. 
\par
Any $\bK$-expansion $\theta$ induces a $1$-cocycle $\tau^\theta_1$ on the mapping class group $\Gamma$ of the surface as follows. 
For the rest of this section we write simply $H = \HK$, on which the mapping class group $\Gamma$ acts naturally. We denote by $|\varphi| \in \mathrm{Aut}(H)$ the automorphism induced by an element $\varphi \in \Gamma$. Moreover we denote $H^* := H^1(\Sigma; \bK) = \Hom_{\bK}(H, \bK)$. The map $\tau^\theta_1: \Gamma \to H^*\otimes H^{\otimes 2} = \Hom_{\bK}(H, H^{\otimes 2})$ is defined by 
\begin{equation}\label{eq:dfntau1}
\tau^\theta_1(\varphi)[\gamma] = \theta_2(\gamma) 
- |\varphi|^{\otimes 2}\theta_2(\varphi\inv(\gamma))
\end{equation}
for any $\varphi \in \Gamma$ and $\gamma \in \pi$ \cite[Lemma 2.2]{Kaw05}. 
It is a $1$-cocycle, since 
$$
\tau^\theta_1(\varphi\psi)[\gamma] - \tau^\theta_1(\varphi)[\gamma] 
= - |\varphi\psi|^{\otimes 2}\theta_2(\psi\inv(\varphi\inv\gamma))
+ |\varphi|^{\otimes 2}\theta_2(\varphi\inv\gamma)
= |\varphi|^{\otimes 2}\left(\tau^\theta_1(\psi)|\varphi|\inv[\gamma] \right)
$$
for any $\varphi, \psi \in \Gamma$ and $\gamma \in \pi$. 
Its cohomology class equals the extended first Johnson homomorphism introduced by Morita \cite{Mo93} if $2 \in \bK$ is invertible \cite{Kaw05}.
\begin{lem}\label{lem:nutau}
$$
(1\otimes \flat)\circ \tau^\theta_1 = -k_{\xi_\theta}: \Gamma \to H^*\otimes H^{\otimes 2}.
$$
\end{lem}
\begin{proof}
From \eqref{eq:dfntau1} and \eqref{eq:kqq}, we have 
$$
\flat\tau^\theta_1(\varphi)[\gamma] 
= \flat\circ\theta_2(\gamma) - \flat\circ\theta_2(\varphi\inv(\gamma))
= q_{\xi_\theta}(\gamma) - q_{\xi_\theta}(\varphi\inv(\gamma))
= - k_\xi(\varphi)[\gamma],
$$
as was to be shown.
\end{proof}

On the other hand, we can consider a $\Gamma$-equivariant map  $\mathfrak{c}: H^*\otimes H^{\otimes 2} \to H$ defined by $\mathfrak{c}(f\otimes X_1X_2) = f(X_1)X_2$ for any $f \in H^*$ and $X_1, X_2 \in H$. 
Since $\tau^\theta_1$ is a $1$-cocycle, $\mathfrak{c}\circ \tau^\theta_1: \Gamma \to H$ is also a $1$-cocyle. We have the inclusion homomorphism $j^*: H^1(\Sigma, \pa\Sigma; \bK) \to H^1(\Sigma; \bK)$ with the Poincar\'e duality map
\begin{equation}\label{eq:vtheta}
\jmath: H = H_1(\Sigma; \bK) \cong H^1(\Sigma, \pa\Sigma; \bK) \overset{j^*}\to  H^1(\Sigma; \bK) = H^*,
\end{equation}
which maps $X\in H$ to $(Y \mapsto \flat(XY)) \in H^*$. 
We like to compare the two $1$-cocycles $(1\otimes \flat)\tau^\theta_1$ and $\jmath\circ\mathfrak{c}\circ \tau^\theta_1$ on the mapping class group $\Gamma$ with values in $H^*$. 
As will be shown in \S\ref{sec:relf}, however, we need to consider a groupoid version of expansions for the surface whose boundary is not connected.
Our previous work \cite[\S7.2]{KK15} already introduced a groupoid expansion $\theta$. 
But, for any $1 \leq j \leq n$, it satisfies $\theta_2(\pa_j\Sigma) = \frac12 [\pa_j\Sigma]^{\otimes 2}$, $C\circ\theta_2(\pa_j\Sigma) = 0$ and so $\rot_{\xi_\theta}(\pa_j\Sigma) = -1$. 
On the other hand, Taniguchi \cite{Tan25c} shows that any possible $\bK$-framing can be induced by a genus $g$ Gonzalez-Drinfeld associator for any field $\bK$ of characteristic $0$. 
It is our next problem to find an appropriate formulation of groupoid expansions. \par
\medskip
Thus we confine ourselves to the case $\Sigma = \Sigma_{g,1}$ for the rest of this section. But we consider an arbitrary commutative ring $\bK$ with unit. \par
Let $\zeta := [\pa\Sgone]\inv \in \pi$ be the negative boundary loop.
A {\it symplectic generating system} $\{\alpha_i, \beta_i\}^g_{i=1} \subset \pi$ is a free generating system of the free group $\pi$ satisfying the equation $\zeta = \prod^g_{i=1}\alpha_i\beta_i{\alpha_i}\inv{\beta_i}\inv \in \pi$.
Symplectic generator systems do exist. 
If we fix a symplectic generator system  $\{\alpha_i, \beta_i\}^g_{i=1}$, 
we denote $[\alpha_i] = A_i$ and $[\beta_i] = B_i \in H$. We have $\flat(A_iB_j) = \delta_{ij}$ and  $\flat(A_iA_j) = \flat(B_iB_j) = 0$ for any $1 \leq i, j \leq g$. 
We denote by $\{A^*_i, B^*_i\}^g_{i=1} \subset H^*$ the dual basis of $\{A_i, B_i\}^g_{i=1}$. 
\begin{lem}\label{lem:thze} For any $\bK$-expansion $\theta$ we have
$$
\theta_2(\zeta) 
= \sum^g_{i=1} A_iB_i - B_iA_i = \sum^g_{i=1} [A_i, B_i] 
$$
and 
\begin{equation}\label{eq:thz3}
\aligned
 \theta_3(\zeta) 
& = \sum^g_{i=1} \theta_2(\alpha_i)B_i + A_i\theta_2(\beta_i)
- \theta_2(\beta_i)A_i - B_i\theta_2(\alpha_i)
-(A_iB_i - B_iA_i)(A_i+B_i)\\
& = \sum^g_{i=1} [A_i, \theta_2(\beta_i) - B_iA_i]
- [B_i, \theta_2(\alpha_i) - A_iB_i].
\endaligned
\end{equation}
Here $[X, Y] = XY - YX \in H^{\otimes 2}$ is the commutator of $X$ and $Y \in H$. 
\end{lem}
\begin{proof} The maps $\theta_2$ and $\theta_3$ are group homomorphisms on the commutator subgroup $[\pi, \pi]$. Hence it suffices to compute $\theta_2$ and $\theta_3$ for a single commutator $\gamma_1\gamma_2{\gamma_1}\inv{\gamma_2}\inv$ of $\gamma_1, \gamma_2 \in \pi$. 
Only inside this proof, we write simply $\Gamma_1 := [\gamma_1]$ and $\Gamma_2 := [\gamma_2] \in H$. 
We look at the equation $\theta(\gamma_1\gamma_2) = \theta(\gamma_1\gamma_2{\gamma_1}\inv{\gamma_2}\inv\gamma_2\gamma_1)$, which we will compute explicitly as 
$$
\theta(\gamma_1\gamma_2) \equiv 
1 + (\Gamma_1+\Gamma_2) + \theta_2(\gamma_1\gamma_2) + \theta_3(\gamma_1\gamma_2)
$$
and so 
$$
\aligned
&\theta(\gamma_1\gamma_2{\gamma_1}\inv{\gamma_2}\inv\gamma_2\gamma_1)\\
& \equiv 1+ (\Gamma_1 + \Gamma_2) 
+ \theta_2(\gamma_1\gamma_2{\gamma_1}\inv{\gamma_2}\inv) +\theta_2(\gamma_2\gamma_1)\\
&\quad + \theta_3(\gamma_1\gamma_2{\gamma_1}\inv{\gamma_2}\inv) 
+ \theta_2(\gamma_1\gamma_2{\gamma_1}\inv{\gamma_2}\inv)(\Gamma_1+\Gamma_2)
+\theta_3(\gamma_2\gamma_1).
\endaligned
$$
Hence we obtain 
\begin{equation*}\label{eq:theta2}
\theta_2(\gamma_1\gamma_2{\gamma_1}\inv{\gamma_2}\inv) 
= \theta_2(\gamma_1\gamma_2) - \theta_2(\gamma_2\gamma_1)
= \Gamma_1\Gamma_2 - \Gamma_2\Gamma_1
= [\Gamma_1, \Gamma_2]
\end{equation*}
and 
\begin{equation*}\label{eq:theta3}
\aligned
&\theta_3(\gamma_1\gamma_2{\gamma_1}\inv{\gamma_2}\inv) 
= \theta_3(\gamma_1\gamma_2) - \theta_3(\gamma_2\gamma_1)
- \theta_2(\gamma_1\gamma_2{\gamma_1}\inv{\gamma_2}\inv)(\Gamma_1+\Gamma_2)\\
&= \theta_2(\gamma_1)\Gamma_2 + \Gamma_1\theta_2(\gamma_2)
- \theta_2(\gamma_2)\Gamma_1 - \Gamma_2\theta_2(\gamma_1)
-(\Gamma_1\Gamma_2 - \Gamma_2\Gamma_1)(\Gamma_1+\Gamma_2)\\
&= [\theta_2(\gamma_1), \Gamma_2] + [\Gamma_1, \theta_2(\gamma_2)]
- [\Gamma_1, \Gamma_2\Gamma_1] - [\Gamma_1\Gamma_2, \Gamma_2]\\
&= [\Gamma_1, \theta_2(\gamma_2) - \Gamma_2\Gamma_1]
- [\Gamma_2, \theta_2(\gamma_1) - \Gamma_1\Gamma_2].
\endaligned
\end{equation*}
This proves the lemma.
\end{proof}

Massuyeau \cite{Mas12} introduced the notion of a symplectic expansion for $\Sgone$ in the case $\bK$ is a field of characteristic $0$. 
We weaken his condition to get the following notion for any commutative ring $\bK$ with unit. 
A $\bK$-expansion $\theta: \pi \to \widehat{T}(H)$ is {\it weakly $3$-symplectic} if it satisfies the condition
\begin{equation}\label{eq:wsymp}
\theta_3(\zeta) = 0 \in {\HK}^{\otimes 3}.
\end{equation}
Here it should be remarked that a weakly $3$-symplectic expansion does not respect the coproduct of the complete Hopf algebra $\widehat{T}(H)$ in general. 
Weakly $3$-symplectic $\bK$-expansions do exist for any commutative ring $\bK$ with unit. In fact, from the second equation of \eqref{eq:thz3}, we may take $\theta_2(\alpha_i) = A_iB_i$ and $\theta_2(\beta_i) = B_iA_i$ for any $1 \leq i \leq g$, for example.\par
\begin{thm}\label{thm:tauc}
If a $\bK$-expansion $\theta: \pi \to \widehat{T}(\HK)$ is weakly $3$-symplectic, then we have
$$
(1\otimes \flat)\tau^\theta_1 = \jmath\circ\mathfrak{c}\circ \tau^\theta_1: \Gamma \to H^*.
$$
\end{thm}

This theorem together with Lemma \ref{lem:nutau} implies 
\begin{cor}\label{cor:ktauc}
If a $\bK$-expansion $\theta: \pi \to \widehat{T}(\HK)$ is weakly $3$-symplectic, then we have
$$
\jmath\circ\mathfrak{c}\circ \tau^\theta_1 = -k_{\xi_\theta}: \Gamma \to H^*.
$$
\end{cor}

The origin of the corollary is in Johnson's theorem \cite[Theorem 2]{J80b} on the Torelli group. Later Salter \cite[\S5]{Sal} refined the result to relate it to $r$-spin structures for $r \in \bZ$. 
Morita \cite{Mo93} extended the first Johnson homomorphism to the whole mapping class group, and generalized Johnson's theorem to this situation \cite[Remark 4.9]{Mo93}, where the condition $\frac12 \in \bK$ was needed, since the composite of the commutator and the quotient $\La^2H \overset{[\cdot, \cdot]}\to H^{\otimes 2} \to \La^2H$ equals twice the identity. 
All the these works use the contraction map $C: X\wedge Y \wedge Z  \in\La^3H \mapsto (X\cdot Y)Z + (Y\cdot Z)X + (Z\cdot X)Y \in H$. 
But, in this paper,  we use another map $\mathfrak{c}$, which enables us to cover all the commutative rings $\bK$ with unit. 
If $\frac12\in \bK$, it coincides with the map $C$ on $\La^3H$. 
We can carry out a computation similar to the below one for the surface $\Sigma_{g, n+1}$, $g, n \geq 1$. But we need appropriate groupoid variants of $\bK$-expansions and Johnson homomorphisms, in order to make it a mathematical theorem. 
\par
The author \cite{Kaw06} constructed a $\bR$-symplectic expansion $\theta^{(C, P_0, v)} = \sum^\infty_{m=0}\theta^{(C, P_0, v)}_m$ of $\pi_1(\Sgone)$ canonically associated with any triple $(C, P_0, v)$, where $C$ is a closed compact Riemann surface of genus $g$, $P_0 \in C$ is a point, and $v \in T_{P_0}C$ is a nonzero tangent vector. 
The $\bR$-framing $\xi_{\theta^{(C, P_0, v)}}= \flat\circ \theta^{(C, P_0, v)}_2$ on $C \setminus \{P_0\}$ is given by the iterated integral of a connection form canonically associated with the triple $(C, P_0, v)$. 
Hence it is real analytic in $(C, P_0, v)$. 
Moreover $\flat\circ\theta^{(C, P_0, v)}_2$ depends only on $(C, P_0)$ from \cite[Proposition 6.3]{Kaw06}. 
Thus we obtain a real analytic family of $\bR$-framings of $C\setminus \{P_0\}$ parametrized by the moduli of pointed Riemann surfaces $(C, P_0)$.
It would be very interesting if one could find another geometric meaning of the $\bR$-framing. 
\par

\medskip
We devote the rest of this section to the proof of Theorem \ref{thm:tauc}. 
Since both of $(1\otimes \flat)\tau^\theta_1$ and $\jmath\circ\mathfrak{c}\circ \tau^\theta_1$ are $1$-cocycles, it suffices to check the equation for some generators of the group $\Gamma$. 
As was proved by Johnson \cite[Theorem 1, p.434]{J83}, 
the mapping class group $\Gamma$ is generated by the right-handed Dehn twists $t_C$ along non-separating simple closed curves $C$ and the right-handed Dehn twist along a simple closed curve parallel to the boundary $\pa\Sgone$. 
The cocycle $\tau^\theta_1$ vanishes at the latter Dehn twist. 
Hence it suffices to prove that the right-handed Dehn twists $t_C$ along any non-separating simple closed curve $C$ satisfies the given equation. \par
From the classification theorem of surfaces, we can have a symplectic generator system $\{\alpha_i, \beta_i\}^g_{i=1} \subset \pi$ such that $C$ is freely homotopic to $\alpha_1$, so that we will compute $\tau^\theta_1(t_C)$ explicitly.\par
\medskip
In the following computation we suppose $2 \leq i \leq g$ unless otherwise specified. The Dehn twist $t_C$ acts on the fundamental group $\pi$ by
$$
\aligned
&t_C(\alpha_1) = \alpha_1, \quad
t_C(\beta_1) = \beta_1\alpha_1, \quad
t_C(\alpha_i) = \alpha_i, \quad
t_C(\beta_i) = \beta_i,\\
& {t_C}\inv(\alpha_1) = \alpha_1, \quad
{t_C}\inv(\beta_1) = \beta_1{\alpha_1}\inv, \quad
{t_C}\inv(\alpha_i) = \alpha_i, \quad
{t_C}\inv(\beta_i) = \beta_i.
\endaligned
$$
If we denote by $T = |t_C|$ the action of $t_C$ on the homology group $H$, 
we have $T-1 = B^*_1\otimes A_1$ on $H$.
The transpose $T^*$ of $T$ is given by
$T^*-1= A_1\otimes B^*_1$ on $H^*$.
From what we recalled so far, we have 
$$
\aligned
&\tau^\theta_1(t_C)A_i = (1 - T^{\otimes 2})\theta_2(\alpha_i), \quad
\tau^\theta_1(t_C)B_i = (1 - T^{\otimes 2})\theta_2(\beta_i), \quad
\tau^\theta_1(t_C)A_1 = (1 - T^{\otimes 2})\theta_2(\alpha_1), \\
&\tau^\theta_1(t_C)B_1 
= \theta_2(\beta_1) - T^{\otimes 2}\theta_2(\beta_1{\alpha_1}\inv)
= \theta_2(\beta_1{\alpha_1}\inv\alpha_1) - T^{\otimes 2}\theta_2(\beta_1{\alpha_1}\inv)\\
& = (1 - T^{\otimes 2})\theta_2(\beta_1{\alpha_1}\inv) + \theta_2(\alpha_1) + 
(B_1-A_1)A_1.
\endaligned
$$
Consequently we obtain
\begin{equation}\label{eq:tau1tC}
\aligned
\tau^\theta_1(t_C) &= 
\sum^g_{i=2}A^*_i\otimes  (1 - T^{\otimes 2})\theta_2(\alpha_i)
+ \sum^g_{i=2}B^*_i\otimes  (1 - T^{\otimes 2})\theta_2(\beta_i)\\
&\quad + A^*_1\otimes (1-T^{\otimes 2})\theta_2(\alpha_1)\\
&\quad + B^*_1\otimes (1-T^{\otimes 2})\theta_2(\beta_1{\alpha_1}\inv)
+  \sout{B^*_1\otimes \theta_2(\alpha_1)} + B^*_1\otimes (B_1-A_1)A_1.
\endaligned
\end{equation}
\par
Since $T$ preserves the intersection form $\flat$, we have $\flat(1 - T^{\otimes 2}) = 0$, and so 
\begin{equation}\label{eq:nutautC}
(1\otimes\flat)(\tau^\theta_1(t_C))
= (\flat\theta_2(\alpha_1) -1)B^*_1.
\end{equation}

On the other hand, since
$
1- T^{\otimes 2} 
= - 1\otimes (T-1) -(T-1)\otimes 1 - (T-1)\otimes(T-1), 
$
we have 
$$
\aligned
&\mathfrak{c}(f\otimes(1-T^{\otimes 2})(X_1X_2))\\
&=-\mathfrak{c}(f\otimes (X_1(T-1)X_2+((T-1)X_1)X_2 + ((T-1)X_1)((T-1)X_2)))\\
&= -f(X_1)(T-1)X_2 - ((T^*-1)f)(X_1)X_2 - ((T^*-1)f)(X_1)(T-1)X_2\\
&= -f(X_1)B^*_1(X_2)A_1 - f(A_1)B^*_1(X_1)X_2 
- f(A_1)B^*_1(X_1)B^*_1(X_2)A_1\\
&= -(f\otimes B^*_1\otimes 1)(X_1X_2A_1+A_1X_1X_2) 
- (f\otimes B^*_1\otimes B^*_1)(A_1X_1X_2)A_1
\endaligned
$$
for $f \in H^*$ and $X_1, X_2 \in H$. 
Hence we have 
$$
\aligned
& \mathfrak{c}(A^*_i\otimes  (1 - T^{\otimes 2})\theta_2(\alpha_i))\\
& = -(A^*_i\otimes B^*_1\otimes 1)(\theta_2(\alpha_i)A_1 + A_1\theta_2(\alpha_i))
- (A^*_i\otimes B^*_1\otimes B^*_1)(A_1\theta_2(\alpha_i))\\
&= -(A^*_i\otimes B^*_1)(\theta_2(\alpha_i))A_1,\\
& \mathfrak{c}(B^*_i\otimes  (1 - T^{\otimes 2})\theta_2(\beta_i))\\
& = -(B^*_i\otimes B^*_1\otimes 1)(\theta_2(\beta_i)A_1 + A_1\theta_2(\beta_i))
- (B^*_i\otimes B^*_1\otimes B^*_1)(A_1\theta_2(\beta_i))\\
&= -(B^*_i\otimes B^*_1)(\theta_2(\beta_i))A_1,
\endaligned
$$
and
$$
\aligned
& \mathfrak{c}(A^*_1\otimes  (1 - T^{\otimes 2})\theta_2(\alpha_1))\\
& = -(A^*_1\otimes B^*_1\otimes 1)(\theta_2(\alpha_1)A_1 + A_1\theta_2(\alpha_1))
- (A^*_1\otimes B^*_1\otimes B^*_1)(A_1\theta_2(\alpha_1))\\
&= -(A^*_1\otimes B^*_1)(\theta_2(\alpha_1))A_1
- \sout{(B^*_1\otimes 1)\theta_2(\alpha_1)} - \sout{(B^*_1\otimes B^*_1)(\theta_2(\alpha_1))A_1},\\
& \mathfrak{c}(B^*_1\otimes  (1 - T^{\otimes 2})\theta_2(\beta_1{\alpha_1}\inv))\\
& = -(B^*_1\otimes B^*_1\otimes 1)(\theta_2(\beta_1{\alpha_1}\inv)A_1 + A_1\theta_2(\beta_1{\alpha_1}\inv))
- (B^*_1\otimes B^*_1\otimes B^*_1)(A_1\theta_2(\beta_1{\alpha_1}\inv))\\
&= -(B^*_1\otimes B^*_1)(\theta_2(\beta_1{\alpha_1}\inv))A_1\\
&= -(B^*_1\otimes B^*_1)(\theta_2(\beta_1))A_1+ \sout{(B^*_1\otimes B^*_1)(\theta_2(\alpha_1))A_1},
\endaligned
$$
where the last equality follows from $\theta_2(\beta_1) = \theta_2(\beta_1{\alpha_1}\inv\alpha_1) = \theta_2(\beta_1{\alpha_1}\inv) + \theta_2(\alpha_1) + (B_1-A_1)A_1$. 
As was underlined, each of the terms $(B^*_1\otimes 1)(\theta_2(\alpha_1))A_1$ and $(B^*_1\otimes 1)(\theta_2(\alpha_1)$ appears in $\mathfrak{c}\tau^\theta_1(t_C)$ twice with opposite signs. Hence we obtain 
\begin{equation}\label{eq:k1tautC}
\mathfrak{c}\tau^\theta_1(t_C) 
=\left(1 -\sum^g_{i=1}(A^*_i\otimes B^*_1)(\theta_2(\alpha_i))
+ (B^*_i\otimes B^*_1)(\theta_2(\beta_i))\right)A_1.
\end{equation}
Now, from the first line of \eqref{eq:thz3}, we have
\begin{equation*}
\aligned
 (\flat\otimes B^*_1)\theta_3(\zeta) 
&= \flat\theta_2(\alpha_1) - \flat(A_1B_1-B_1A_1)
+ \sum^g_{i=1}(B^*_i\otimes B^*_1)\theta_2(\beta_i) + (A^*_i\otimes B^*_1)\theta_2(\alpha_i) \\
&= \flat\theta_2(\alpha_1) -2
+ \sum^g_{i=1}(B^*_i\otimes B^*_1)\theta_2(\beta_i) + (A^*_i\otimes B^*_1)\theta_2(\alpha_i).\\
\endaligned
\end{equation*}
Here we used $\flat(A_iX) = B^*_i(X)$ and $\flat(B_iX) =  -A^*_i(X)$ for any $X \in H$. 
Therefore, if $\theta$ is weakly $3$-symplectic, then we have
$
\mathfrak{c}\tau^\theta_1(t_C) = (\flat\theta_2(\alpha_1) - 1)A_1,
$
which together with \eqref{eq:nutautC} and $\jmath(A_1) = B^*_1$
implies Theorem \ref{thm:tauc}.\qed
\par
Let $(12): H^{\otimes 2} \to H^{\otimes 2}$ be the switching map. 
From what we computed above, the map $\mathfrak{c}\circ(1\otimes (12)): H^*\otimes H^{\otimes 2} \to H$, $f\otimes (X_1X_2) \mapsto f(X_2)X_1$, does {\it not} satisfy a similar formula to Theorem \ref{thm:tauc}. 
\par

%
%
\section{Relative $\bK$-framings}\label{sec:relf}

In this section, we suppose $\pa\Sigma\neq\emptyset$, and prove that the Earle class $k_\bK \in H^1(\Gamma; H^1(\Sigma; \bK))$ 
has a natural lift with respect to the inclusion homomorphism $j^*$ in the cohomology exact sequence
$$
0 \to \widetilde{H}^0(\pa\Sigma; \bK) \to H^1(\Sigma, \pa\Sigma; \bK) 
\overset{j^*}\to H^1(\Sigma; \bK) 
\overset{i^*}\to H^1(\pa\Sigma; \bK) 
\overset{\delta^*}\to H^2(\Sigma, \pa\Sigma; \bK)
$$
of the pair $(\Sigma, \pa\Sigma)$. If we denote
$$
N_\pa(\Sigma) =  N_\pa(\Sigma; \bK) := \Ker(\delta^*) = \mathrm{Im}(i^*) \subset H^1(\pa\Sigma; \bK),
$$
we have 
\begin{equation}\label{eq:relS}
0 \to \widetilde{H}^0(\pa\Sigma; \bK) \to H^1(\Sigma, \pa\Sigma; \bK) 
\overset{j^*}\to H^1(\Sigma; \bK) 
\overset{i^*}\to N_\pa(\Sigma; \bK) \to 0\quad \text{(exact)}.
\end{equation}
\par
The lift should come from `relative $\bK$-framings' on $(\Sigma, \pa\Sigma)$. 
The notion of a relative framing was introduced by Randal-Williams \cite{RW14}. 
It is an isomorphism $U\Sigma \cong \Sigma\times S^1$ of oriented $S^1$-bundles, which restricts to a fixed isomorphism $\delta: U\Sigma\vert_{\pa\Sigma} \cong \pa\Sigma\times S^1$. 
In this paper, we call its homotopy class {\it a relative framing}. 
More precisely, for each $0 \leq j \leq n$, choose a number $\rho_j \in \bK$ and a single point $\bullet_j \in \pa_j\Sigma$. The vector $(\rho_j)^n_{j=0} \in \bK^{n+1}$
can be identified with a unique element $\rho \in H^1(\pa\Sigma; \bK)$ satisfying 
$\langle\rho, [\pa_j\Sigma]\rangle = \rho_j$ for each $0 \leq j \leq n$. 
Then we have $\rho \in N_\pa(\Sigma; \bK)$ if and only if $\sum^n_{j=0}\rho_j = 0$. 
Now we denote $P := \{\bullet_j\}^n_{j=0}$ and 
$$
F(\Sigma, P, \rho; \bK) := \{\xi \in F(\Sigma, P; \bK); \,\, 0 \leq \forall j \leq n, \, 
\rot_\xi(\pa_j\Sigma) = \rho_j\}, 
$$
whose element we call {\it a relative $\bK$-framing}. 
Here the inclusion $U\Sigma \hookrightarrow (U\Sigma, \tilde P)$ induces a map $F(\Sigma, P; \bK) \to F(\Sigma;\bK)$, so that we may consider the rotation number $\rot_\xi\pa_j\Sigma \in \bK$ for each $0 \leq j \leq n$. 
The evaluation map $\mathrm{Map}(\pa_j\Sigma, S^1) \to S^1$, $f \mapsto f(\bullet_j)$, at $\bullet_j \in \pa_j\Sigma$ is a fibration with fiber $\Omega S^1\simeq \bZ$.
Hence the tangent vector $v_{\bullet_j}$ and the rotation number $\rot_\xi(\pa_j\Sigma) \in \bZ$ describe the homotopy type of the space $\mathrm{Map}(\pa_j\Sigma, S^1)$ for each $0 \leq j \leq n$.
\par
Since $H^*(\pa\Sigma, P; \bK) \cong \prod^n_{j=0}\widetilde{H}^*(\pa_j\Sigma; \bK)$, 
the cohomology exact sequence of the triple $(\Sigma, \pa\Sigma, P)$ induces a short exact sequence
\begin{equation}\label{eq:SdP}
0 \to H^1(\Sigma, \pa\Sigma; \bK) \overset{j^*_P}\to H^1(\Sigma, P; \bK) \overset{i^*_P}\to N_\pa(\Sigma; \bK) \to 0\quad \text{(exact)}.
\end{equation}
\begin{lem}\label{lem:nonempty}
(1) The set $F(\Sigma, P, \rho; \bK)$ is non-empty if and only if $\sum^n_{j=0}\rho_j = \chi(\Sigma) \in \bK$.\par
(2) If the set $F(\Sigma, P, \rho; \bK)$ is non-empty, it is a $H^1(\Sigma, \pa\Sigma; \bK)$-torsor.
\end{lem}
\begin{proof}
(1) The `only if' part follows from Lemma \ref{lem:PH}. 
In order to prove the `if' part, we take an element $\xi_0 \in F(\Sigma; \bK)$, 
since $\pa\Sigma \neq \emptyset$. 
From the exact sequence \eqref{eq:SdP}, 
there is an element $u \in H^1(\Sigma, P; \bK) $ such that $\langle i^*u, [\pa_j\Sigma]\rangle = \rho_j - \rot_{\xi_0}(\pa_j\Sigma)$ for each $0\leq j \leq n$. 
Then we have $\rot_{\xi_0+u}(\pa_j\Sigma) = \rho_j$ for each $0\leq j \leq n$.
This proves (1).\par
(2) Recall the set $F(\Sigma, P; \bK)$ is a $H^1(\Sigma, P; \bK)$-torsor. An element $u \in H^1(\Sigma, P; \bK)$ fixes each $\rot_\xi(\pa_j\Sigma)$ if and only if $i^*_Pu = 0$. This, together with \eqref{eq:SdP}, proves (2).
\end{proof}

For any $\rho \in H^1(\pa\Sigma; \bK)$ with $F(\Sigma, P, \rho; \bK)\neq\emptyset$ and any $\xi \in F(\Sigma, P, \rho; \bK)$, we define a cocycle $k_\xi: \Gamma \to H^1(\Sigma, \pa\Sigma; \bK)$ by
$$
k_\xi(\varphi) = \varphi\cdot \xi - \xi \in H^1(\Sigma, \pa\Sigma;\bK)
$$
for $\varphi \in \Gamma$, and denote by
$$
k_{\rho, \bK} := [k_\xi] \in H^1(\Gamma; H^1(\Sigma, \pa\Sigma;\bK))
$$
its cohomology class, which is independent of the choice of $\xi \in F(\Sigma, P, \rho; \bK)$. 
The inclusion $F(\Sigma, \pa\Sigma; \bK) \hookrightarrow F(\Sigma, P; \bK)$ is equivariant under the homomorphism $j^*_P$, and all of these are $\Gamma$-equivariant. 
Hence we have 
$(j^*_P)_*(k_{\rho, \bK}) = k^\natural_\bK \in H^1(\Gamma; H^1(\Sigma, P; \bK))$ and so 
\begin{equation}\label{eq:blift}
(j^*)_*(k_{\rho, \bK}) = k_\bK \in H^1(\Gamma; H^1(\Sigma; \bK)).
\end{equation}
from \eqref{eq:natlift}. The cohomology class $k_{\rho, \bK} \in H^1(\Gamma; H^1(\Sigma, \pa\Sigma; \bK))$ depends on the choice of $\rho \in H^1(\pa\Sigma; \bK)$. More precisely we have
\begin{lem}\label{lem:crho} Let $\rho$ and $\rho' \in H^1(\pa\Sigma; \bK)$ satisfy $F(\Sigma, P, \rho; \bK)\neq\emptyset$ and $F(\Sigma, P, \rho'; \bK)\neq\emptyset$, respectively. Then we have 
$$
k_{\rho', \bK} - k_{\rho, \bK} = \delta^*(\rho'-\rho) \in H^1(\Gamma; H^1(\Sigma, \pa\Sigma;\bK)),
$$
where $\delta^*$ is the connecting homomorphism with respect to the extension \eqref{eq:SdP} of $\Gamma$-modules.
\end{lem}
\begin{proof} From Lemma \ref{lem:nonempty}, we have $\rho' - \rho \in N_\pa(\Sigma; \bK)$. 
Take $\xi \in F(\Sigma, P, \rho; \bK)$ and $\xi' \in F(\Sigma, P, \rho'; \bK)$ arbitrarily. Then the map $i^*_P$ maps $u := \xi'-\xi \in H^1(\Sigma, P; \bK)$ to $\rho' - \rho$. 
Hence we have $(\delta u)(\varphi) = \varphi\cdot u - u = (\varphi\cdot \xi' - \xi') - (\varphi\cdot\xi - \xi) = (k_{\xi'}-k_\xi)(\varphi)$ for any $\varphi \in \Gamma$. This proves the lemma. 
\end{proof}

As was proved in Lemma \ref{lem:coinvariants} (4), we have a natural isomorphism $\delta^*: \widetilde{H}^0(\pa\Sigma; \bK) \overset\cong \to H^1(\Sigma, \pa\Sigma; \bK)^\Gamma$. 
In particular, the subset $k_{\rho,\bK}\subset Z^1(\Gamma; H^1(\Sigma, \pa\Sigma; \bK))$ is not a $H^1(\Sigma, \pa\Sigma; \bK)$-torsor if $\pa\Sigma$ is not connected. Instead, we have the following.
\begin{lem}\label{lem:cf} We have a natural map $\zeta: H^1(\Sigma; \bK) \to Z^1(\Gamma; H^1(\Sigma, \pa\Sigma; \bK))$ which induces a morphism of exact sequences
$$
\begin{xymatrix}{
0 \ar[r] & \widetilde{H}^0(\pa\Sigma) \ar[r]^{\delta^*} \ar@{=}[d]& H^1(\Sigma, \pa\Sigma) \ar[r]^{j^*} \ar@{=}[d]& H^1(\Sigma)  \ar[r]^{i^*} \ar[d]^{\zeta}& N_\pa(\Sigma) \ar[r] \ar[d]^{\delta^*}& 0\\
0 \ar[r] & \widetilde{H}^0(\pa\Sigma) \ar[r]^{\delta^*} & H^1(\Sigma, \pa\Sigma) \ar[r]^-{\delta} & Z^1(\Gamma; H^1(\Sigma, \pa\Sigma))  \ar[r] & H^1(\Gamma; H^1(\Sigma, \pa\Sigma)) \ar[r] & 0,\\
}\end{xymatrix}
$$
where we omit the coefficients $\bK$. 
\end{lem}
\begin{proof} 
We can identify $H^1(\Sigma) = H_1(\Sigma, \pa\Sigma)$ and $H^1(\Sigma, \pa\Sigma) = H_1(\Sigma)$ by the Poincar\'e duality. 
Hence it suffices to show that the map
$$
\zeta: 
u \in H_1(\Sigma, \pa\Sigma) \mapsto (\varphi \in \Gamma \mapsto \varphi \cdot u - u) \in Z^1(\Gamma; H_1(\Sigma))
$$
is well-defined. 
From the homology exact sequence of the pair $(\Sigma, P)$
$$
0 \to H_1(\Sigma) \to H_1(\Sigma, P) \overset{\pa_*}\to \widetilde{H}_0(P) \to 0 \quad (\text{exact}),
$$
we have $\pa_*(\varphi\cdot u - u) = 0$ for any $u \in H_1(\Sigma, P)$. 
Hence we may regard $\varphi\cdot u - u \in H_1(\Sigma)$. 
Then look at the homology exact sequence of the triple $(\Sigma, \pa\Sigma, P)$
$$
H_1(\pa\Sigma, P) \to H_1(\Sigma, P) \to H_1(\Sigma, \pa\Sigma) \to 0
\quad (\text{exact}).
$$
The homology group $H_1(\pa\Sigma, P)$ is generated by $[\pa_j\Sigma]$, $0 \leq j \leq n$, and $\varphi [\pa_j\Sigma] - [\pa_j\Sigma] = 0$. 
Hence the map 
$$
u \in H_1(\Sigma, \pa\Sigma) \mapsto \varphi\cdot u - u \in H_1(\Sigma)
$$
is well-defined. Now it is clear that $\varphi\cdot u - u$ is a cocycle with respect to $\varphi\in \Gamma$, and that the diagram commutes. This proves the lemma. 
\end{proof}
A cocycle similar to the cocycle $\varphi \in \Gamma \mapsto \varphi\cdot u - u \in H_1(\Sigma)$ appeared in our previous work \cite[\S3]{Kaw98}. 
\par
We conclude this paper by computing some stable twisted cohomology groups. 
By Apisa and Salter's computation \cite[Theorem A.4]{ApisaSalter}, 
the following lemma also holds for $\Sigma = \Sigma_{g, n+1}$ with $g \geq 3$ and $n+1 \geq 3$. 
\begin{lem}\label{lem:compute} Let $\bK$ be a commutative ring with unit, and suppose $g \geq 5$. Then\par
(1) $H^1(\Gamma; H^1(\Sigma; \bK)) = \bK$, and it is generated by the class $k_\bK$. \par
(2) $H^1(\Gamma; H^1(\Sigma, P; \bK)) = \bK$, and it is generated by the class $k^\natural_\bK$. 
\par
(3) We have an exact sequence
$$
0 \to N_\pa(\Sigma; \bK) \overset{\delta^*}\to 
H^1(\Gamma; H^1(\Sigma, \pa\Sigma; \bK)) \overset{(j^*)_*}\to 
H^1(\Gamma; H^1(\Sigma; \bK)) \to 0 \quad (\text{exact}).
$$
\end{lem}
\begin{proof} (1) Since $g \geq 5$, we have $H_1(\Gamma; H_1(\Sigma;\bZ)) = \bZ$ from \cite[Theorem 7.1]{H91}. We glue $\Sigma_{g,1}$ and $\Sigma_{0,n+2}$ along the boundary of $\Sigma_{g,1}$ to get the surface $\Sigma = \Sigma_{g,n+1}$, and cap $n$ disks along $\pa_j\Sigma$ for $1 \leq j \leq n$ to get another $\Sigma_{g,1}$. These inclusions induce a sequence 
$$
H_1(\Gamma_{g,1}; H_1(\Sigma_{g,1}; \bZ))
\to H_1(\Gamma; H_1(\Sigma; \bZ))
\to H_1(\Gamma_{g,1}; H_1(\Sigma_{g,1}; \bZ))
$$
whose composite is the identity. Since these three cohomology groups are isomorphic to $\bZ$, each of the maps in the sequence is an isomorphism. 
\par
From the universal coefficient theorem, we have an exact sequence
$$
0 \to \mathrm{Ext}^1_{\bZ}(H_1(\Sigma; \bZ)_\Gamma, \bK) 
\to H^1(\Gamma; H^1(\Sigma; \bK))
\to \Hom_\bZ(H_1(\Gamma; H_1(\Sigma; \bZ)), \bK) \to 0
\quad \text{(exact)}.
$$
In fact, if $P_* \to \bZ \to 0$ be a $\bZ[\Gamma]$-free resolution, then we have a natural isomorphism $\Hom_{\bZ[\Gamma]}(P_*, H^1(\Sigma; \bK))
= \Hom_\bZ(P_*\otimes_{\bZ[\Gamma]}H_1(\Sigma; \bZ), \bK)$. See, for example, \cite[Lemma 2.4]{KS1}, where we consider a slightly different situation. \par
As was shown in Lemma \ref{lem:coinvariants} (1), 
$H_1(\Sigma; \bZ)_\Gamma$ is $\bZ$-free. 
Consequently the inclusion $\Sigma \hookrightarrow \Sigma_{g,1}$ induces an isomorphism $H^1(\Gamma_{g,1}; H^1(\Sigma_{g,1}; \bZ)) \cong 
H^1(\Gamma; H^1(\Sigma; \bZ))$, where the former one is generated by the class $k_{\bZ}$ from Morita's computation \cite{Mo89}. Moreover we have $H^1(\Gamma; H^1(\Sigma; \bK)) = H^1(\Gamma; H^1(\Sigma; \bZ))\otimes\bK$ from what we showed. This proves (1).\par
(2)  We look at the exact sequence 
$$
0 \to \widetilde{H}^0(P; \bK) \to H^1(\Sigma, P; \bK) \to H^1(\Sigma; \bK) \to 0 \quad\text{(exact)},
$$
where the group $\Gamma$ acts on $\widetilde{H}^0(P; \bK)$ trivially. 
By Harer's computation \cite[Lemma 1.1]{H83}, we have $H_1(\Gamma; \bZ) = 0$ since $g \geq 3$. Hence we have 
$$
0 = H^1(\Gamma; \widetilde{H}^0(P; \bK)) \to 
H^1(\Gamma; H^1(\Sigma, P; \bK)) \to 
H^1(\Gamma; H^1(\Sigma; \bK)) = \bK \quad \text{(exact)},
$$
where the generator $k_\bK$ lifts to $H^1(\Gamma; H^1(\Sigma, P; \bK))$ by \eqref{eq:natlift}. This proves (2). \par
(3) Since $g \geq 3$ and the group $\Gamma$ acts on $N_\pa(\Sigma; \bK)$ trivially, the sequence \eqref{eq:SdP} induces a long exact sequence
\begin{equation}\label{eq:exactP}
H^1(\Sigma, \pa\Sigma)^\Gamma \to H^1(\Sigma, P)^\Gamma \to 
N_\pa(\Sigma) \to H^1(\Gamma; H^1(\Sigma, \pa\Sigma))
\to H^1(\Gamma; H^1(\Sigma, P)) \to 0\quad \text{(exact)},
\end{equation}
where we omit the coefficients $\bK$. 
Now we have $H_1(\Sigma; \bZ)_\Gamma = 0$ from Lemma \ref{lem:coinvariants} (1). From the right exactness of the functor $H_0(\Gamma; -)$, we have isomorphisms $H_1(\Sigma, \pa\Sigma; \bZ)_\Gamma \cong \widetilde{H}_0(\pa\Sigma; \bZ)_\Gamma$ and $H_1(\Sigma, P; \bZ)_\Gamma \cong \widetilde{H}_0(P; \bZ)_\Gamma$, while we have an isomorphism $\widetilde{H}_0(\pa\Sigma; \bZ)_\Gamma = \widetilde{H}_0(P; \bZ)_\Gamma$. 
Hence the left arrow in \eqref{eq:exactP} is an isomorphism. This together with (2) proves (3). 
\end{proof}

In this paper we confine ourselves to surfaces without punctures.
In the forthcoming paper \cite{KS4} jointly with A.\ Souli\'e, we will discuss stable twisted cohomology groups of the mapping class groups of surfaces with punctures. \par

%
%

\end{document}